\documentclass[11pt,a4paper]{article}
\usepackage{amssymb,amsmath,amsthm}
\usepackage{color}
\usepackage{marginnote}
\usepackage{comment}
\usepackage{dutchcal}
\usepackage{enumitem}
\usepackage{bm}

\numberwithin{equation}{section}
\newtheorem{thm}{Theorem}[section]
\newtheorem{df}[thm]{Definition}
\newtheorem{prop}[thm]{Proposition}
\newtheorem{lem}[thm]{Lemma}
\newtheorem{cor}[thm]{Corollary}
\newtheorem{rem}[thm]{Remark}

\newtheorem{ex}[thm]{Example}

\newtheorem{hyp}[thm]{Hypotheses}
\let\oldproofname=\proofname
\renewcommand{\proofname}{\rm\bf{\oldproofname}}

\newcommand{\N}{\mathbb{N}}
\newcommand{\Z}{\mathbb{Z}}

\newcommand{\R}{\mathbb{R}}
\newcommand{\C}{\mathbb{C}}
\newcommand{\B}{\mathbb{B}}

\newcommand{\cA}{\mathcal{A}}
\newcommand{\cB}{\mathcal{B}}
\newcommand{\cC}{\mathcal{C}}

\newcommand{\cE}{\mathcal{E}}
\newcommand{\cF}{\mathcal{F}}
\newcommand{\cG}{\mathcal{G}}

\newcommand{\cK}{\mathcal{K}}
\newcommand{\cL}{\mathcal{L}}

\newcommand{\cO}{\mathcal{O}}
\newcommand{\cP}{\mathcal{P}}
\newcommand{\cQ}{\mathcal{Q}}
\newcommand{\cR}{\mathcal{R}}

\newcommand{\cX}{\mathcal{X}}

\renewcommand{\Re}{\mathop{\mathrm{Re}}}

\newcommand{\dd}{\,{\rm d}}
\newcommand{\D}{{\rm d}}
\renewcommand{\div}{\mathop{\mathrm{div}}\nolimits}
\newcommand{\curl}{\mathop{\mathrm{curl}}}

\newcommand{\1}{\mathbf{1}}
\newcommand{\rs}{\mathrm{rs}}
\newcommand{\spec}{\mathrm{spec}}
\newcommand{\wtB}{\widetilde{B}}
\newcommand{\wtL}{\widetilde{L}}
\newcommand{\TS}{\textstyle}
\newcommand{\DS}{\displaystyle}
\newcommand{\QED}{\mbox{}\hfill$\Box$}
\renewcommand{\:}{\thinspace :}

\newcommand{\vf}{\varphi}
\newcommand{\ve}{\varepsilon}

\newcommand{\om}{\omega}
\newcommand{\Xc}{X_c}
\newcommand{\TXc}{T_{\bar x} \Xc}
\newcommand{\irt}{\int_{\R^2}}
\newcommand{\ch}{\mathcal{h}}
\newcommand{\ck}{\mathcal{k}}

\begin{document}

\title{Arnold's variational principle and its application\\ to the 
stability of planar vortices}

\author{
{\bf Thierry Gallay and Vladim\'ir \v{S}ver\'ak}}


\maketitle

\begin{abstract}
We consider variational principles related to V.~I.~Arnold's stability
criteria for steady-state solutions of the two-dimensional incompressible 
Euler equation. Our goal is to investigate under which conditions 
the quadratic forms defined by the second variation of the associated
functionals can be used in the stability analysis, both for the Euler 
evolution and for the the Navier-Stokes equation at low viscosity.
In particular, we revisit the classical example of Oseen's vortex, 
providing a new stability proof with stronger geometric flavor.
Our analysis involves a fairly detailed functional-analytic study of the 
inviscid case, which may be of independent interest, and a careful 
investigation of the influence of the viscous term in the particular 
example of the Gaussian vortex.
\end{abstract}

\section{Introduction}\label{sec1}
In this paper we investigate the applicability of V.~I.~Arnold's geometric
methods to certain stability problems related to Navier-Stokes vortices at high
Reynolds number. Our main goal is a ``proof of concept'' that such applications
are possible, at least in simple cases, even though much of the geometric
structure behind the inviscid stability analysis does not survive the addition
of the viscosity term. In particular, we give a new proof of a known result
concerning the stability of Oseen's vortex as a steady state of the
Navier-Stokes equation in self-similar variables. We expect that the approach we
advertise here will be useful to tackle stability problems involving solutions
that are less symmetric and less explicit than the classical Oseen vortex. In
such cases one may not have good alternative methods for proving stability in
the presence of viscosity. Our investigation leads to a detailed study of the
quadratic forms naturally arising in Arnold's approach. Some of their
functional-analytic properties, which are established in the course of our
analysis, may be of independent interest.

\subsection{A finite-dimensional model}\label{ssec11}

Following V.~I.~Arnold's seminal paper~\cite{Arn1}, we first illustrate the issues we 
want to address in a model situation where the ``phase space'' is finite-dimensional. 
We consider the ordinary differential equation
\begin{equation}\label{I1}
  \dot x \,=\, b(x)\,, \qquad x \in \R^n\,,
\end{equation}
where $b$ is a smooth vector field in $\R^n$. Let us assume that $f,g_1,\dots, g_m
: \R^n \to \R$ are (sufficiently smooth) conserved quantities for the evolution~\eqref{I1}, 
with $m < n$. This means
\begin{equation}\label{I2}
  f'(x)b(x) \,=\, 0 \quad \hbox{and}\quad g_j'(x)b(x) \,=\,0\,\,,\quad x\in\R^n\,,
  \quad j=1,\dots, m\,,
\end{equation}
where we adopt the standard notation $f'(x)$ for the linear form given by
the first derivative of $f$ at $x$.  The situation we have ultimately in mind is
somewhat more specific: we can assume that the phase space $\R^n$ is equipped
with a Poisson structure, that system~\eqref{I1} is of the form
\begin{equation}\label{I3}
    \dot x \,=\, \{f,x\}\,,
\end{equation}
where $\{\,\cdot\,,\cdot\,\}$ denotes the Poisson bracket, and that $g_1,\dots,g_m$
are Casimir functions. This additional structure may of course be important in
many respects, but for our arguments here it will not play a big role. We can
therefore proceed in the context of~\eqref{I1} and~\eqref{I2}.

For any $c = (c_1,\dots,c_m) \in \R^m$, let us denote $\Xc = \{x\in\R^n\,;\,g_1(x)=c_1, 
\dots,g_m(x)=c_m\}$. We assume that, for some $c \in \R^m$, the function $f$ attains 
a {\em non-degenerate local maximum} on $\Xc$ at some point $\bar x\in\Xc$, and that 
the derivatives $g_1'(\bar x), \dots, g_m'(\bar x)$ are linearly independent. The 
stationarity condition at $\bar x$ gives the linear relation
\begin{equation}\label{I4}
  f'(\bar x) - \lambda_1g'_1(\bar x) - \ldots -\lambda_m g_m'(\bar x) \,=\, 0\,,
\end{equation}
for some Lagrange multipliers $\lambda_1,\dots,\lambda_m \in \R$. Moreover, 
the second order differential\footnote{We recall that the second order differential
of a function on a manifold is intrinsically defined at the points where the first
order differential vanishes.} of the function $f|_{\Xc}$ (the restriction of $f$ to 
$\Xc$) at $\bar x$ is given by the restriction to the tangent space $\TXc$ of the 
quadratic form
\begin{equation}\label{I5}
  \cQ \,=\, f''(\bar x) - \lambda_1g''_j(\bar x) - \ldots - \lambda_m g_m''(\bar x)\,,
\end{equation}
where we denote by $f''(\bar x)$ the quadratic form given by the Hessian of $f$
at $\bar x$. Our non-degeneracy assumption means that the restriction of the
form $\cQ$ to $\TXc$ is strictly negative definite. Now, let $B = b'(\bar x)$ be 
the $n \times n$ matrix corresponding to the linearization of~\eqref{I1} at 
the point $\bar x$, which is a steady state by construction. If we differentiate 
twice the relations~\eqref{I2} and use~\eqref{I4} together with $b(\bar x)=0$, 
we see that the evolution defined by the linearized equation $\dot\xi = B\xi$ 
leaves the form $\cQ$ invariant. In other words,
\begin{equation}\label{I7}
  \frac{\D}{\D t} \cQ(\xi,\xi) \,=\, \cQ(B\xi,\xi) + \cQ(\xi,B\xi) \,=\, 0\,,
  \qquad \forall\,\xi \in \R^n\,.
\end{equation}

The above structure\footnote{Pointed out in \cite{Arn1} in the form we use here,
  although in the finite-dimensional case these ideas go back to the founders of
  the analytical mechanics.} gives various options for the stability analysis of
the equilibrium $\bar x$ of~\eqref{I1}, depending on the index of the quadratic
form $\cQ$ in \eqref{I5}. Our assumptions readily imply that $\bar x$ is stable
in the sense of Lyapunov with respect to perturbations on the invariant
submanifold $\Xc$.  Moreover, since a neighborhood of $\bar x$ in $\R^n$ is
foliated by submanifolds of this form for nearby values of the parameter
$c = (c_1,\dots,c_m)$, one can show that $\bar x$ is in fact Lyapunov stable
with respect to small {\em unconstrained} perturbations. The perspective changes
qualitatively if we add to the vector field $b$ in \eqref{I1} a small
``dissipative'' perturbation, with the effect that the quantities $f$ and
$g_1, \dots,g_m$ are no longer exactly conserved under the perturbed
evolution. This is in the spirit of what we intend to do in the
infinite-dimensional case, when we consider the Navier-Stokes equation as a
perturbation of the Euler equation. Since the evolution no longer takes place on
the manifolds $\Xc$, the argument above leading to unconstrained Lyapunov
stability is not applicable anymore. However, in good situations, stability can
still be obtained if the quadratic form $\cQ$ in \eqref{I5} happens to be
negative definite not just on $\TXc$, but on larger subspaces as well, for
instance on the whole space $\R^n$. This is, roughly speaking, the idea we shall
pursue in the infinite-dimensional case, to study the stability of vortex-like
solutions of the Navier-Stokes equation.

To conclude with the (unperturbed) evolution \eqref{I1}, we emphasize that the
problem of determining the index of the form~\eqref{I5} is also very natural
from the viewpoint of the usual constrained maximization/minimization.  Clearly,
the ``Lagrange function''
\begin{equation}\label{I9}
  \cL(x) \,=\, f(x) -\lambda_1g_1(x) -\ldots-\lambda_m g_m(x)\,, \qquad 
  x \in \R^n\,,
\end{equation}
when considered on the whole space $\R^n$, has a critical point at $\bar x$ (and
a local maximum at $\bar x$ when restricted to $\Xc$). The form $\cQ$
will be strictly negative definite\footnote{Our use of the terms ``positive
  definite'' and ``negative definite'' allows for vanishing along some
  directions. When this is not the case, we speak of strictly positive definite 
  or strictly negative definite forms.} in the whole space $\R^n$ if and only 
if $\cL$ has a non-degenerate {\em unconstrained} maximum at $\bar x$. As is 
explained in Section~\ref{ssec24}, this is related to the concavity of the function
\begin{equation}\label{Mf}
  (c_1,\dots,c_m) \,\longmapsto\, M(c_1,\dots, c_m) ~:=~ \sup_{x \in \Xc} f(x)\,.
\end{equation}

\subsection{Arnold's geometric view of the 2d incompressible Euler equation}
\label{ssec12}

V.~I.~Arnold~\cite{Arn2,Arn3,Arnold-Khesin} carried out the analogue of the above
calculations in an infinite-dimensional setting to handle in particular the 2d
incompressible Euler equation $\partial_t \omega + u\cdot\nabla \omega = 0$,
where $u$ denotes the velocity of the fluid and $\omega = \curl u$ is the
associated vorticity. In this case the evolution is generated by the Hamiltonian
function, which represents the kinetic energy of the fluid, and the constraints
are given by the Casimir functionals
\begin{equation}\label{I10}
  \cC_\Phi(\om) \,=\, \int_{\Omega}\Phi(\om(x))\dd x\,,
\end{equation}
where $\Omega \subset \R^2$ is the fluid domain and $\Phi$ is an ``arbitrary''
function on $\R$. The idea of maximizing or minimizing the energy on the set of
vorticities satisfying suitable constraints has been widely used since then to
study the stability of steady-state solutions of the 2d Euler equations and
related fluid models, see \cite{Arnold-Khesin,Burton} and the references 
therein. 

Let us briefly recall the setup relevant for our goals here, making the
similarities with the finite-dimensional case as transparent as possible. 
Our main objects will be the following:

\begin{itemize}[leftmargin=*]

\item[(i)]
The {\em phase space} $\cP=\{\om\colon\R^2\to(0,\infty)\,;\,\hbox{$\om$ is smooth 
and decays ``sufficiently fast" at $\infty$}\}$. This is our infinite-dimensional
replacement for the manifold $\R^n$ in the finite-dimensional model. We restrict
ourselves to positive vorticity distributions defined on $\Omega = \R^2$,
because this is the appropriate framework to study the stability of radially
symmetric vortices in the whole plane. Admittedly, the definition above is somewhat
vague, but it serves only as a motivation and our results will be independent of
the vague parts of the definitions. There is a natural Poisson structure on
$\cP$ that is relevant for the Euler equation, see Section~\ref{ssecA5},
but here we only need some of its Casimir functionals (to be specified now).

\item[(ii)] 
The {\em Casimir functionals}, which play the role of the constraints $g_j$ 
in the finite-dimensional example. These are linear combinations of elementary 
functionals of the form 
\begin{equation}\label{ni1}
  h(a,\om) \,=\, |\{\om>a\}| \,=\, \irt \chi\bigl(\om(x)-a\bigr)\dd x\,,
  \qquad a > 0\,,
\end{equation}
where $\chi = \1_{(0,\infty)}$ is the indicator function of $(0,\infty)$.  Here
and in what follows, we denote by $|Z|$ the Lebesgue measure of any (Borel) set
$Z \subset \R^2$. Due to our assumptions on the vorticities in $\cP$, the
functions $a\mapsto h(a,\om)$ are finite and decreasing on $(0,\infty)$. In
general, they do not have to be continuous in $a$ but they will have this
property in the examples considered later. Similarly, the functionals
$\om\mapsto h(a,\om)$ may in general not be differentiable in every direction,
but they will be in our examples. It is useful to single out the quantity
\begin{equation}\label{M0}
  M_0(\om) \,=\, \irt \om(x)\,\dd x \,=\, \int_0^\infty h(a,\om)\dd a\,,
\end{equation}
which will be referred to as the ``mass'' of the vorticity distribution $\omega
\in \cP$. 

\item[(iii)] The {\em orbits} defined for any $\bar\om \in \cP$ by
\begin{equation}\label{ni2}
  \cO_{\bar\om} \,=\, \bigl\{\om\in\cP\,;\, h(a,\om) = h(a,\bar \om) \hbox{ for all }
  a\in(0,\infty)\bigr\}\,. 
\end{equation}
These subsets of the phase space are the analogues of the manifolds $\Xc$ defined 
by the constraints, and can be considered as a measure-theoretical replacement for 
the symplectic leaves
\[
  \cO_{\bar \om}^{\rm SDiff} \,=\, \bigl\{\om\in\cP\,;\,\om=\bar \om \circ \phi
  \hbox{ for some }\phi\in{\rm SDiff}\bigr\} \,\subset\, \cO_{\bar\om}\,,
\]
where ${\rm SDiff}$ denotes the group of area-preserving diffeomorphisms in $\R^2$. 
In contrast to $\cO_{\bar\om}^{\rm SDiff}$, the orbit $\cO_{\bar \om}$ does not carry 
any topological information about $\bar\om$, since $\om \in \cO_{\bar \om}$ as 
soon as $\om$ is a measure-preserving rearrangement of $\bar\om$. 

\item[(iv)]
The {\em Hamiltonian} (or energy functional) $E\colon\cP\to\R$, given by
\begin{equation}\label{Edef}
  E(\om) \,=\, -\frac 12\irt \psi(x)\om(x)\dd x \,=\, -\frac 1{4\pi} 
  \irt\irt\log|x-y|\,\om(x)\,\om(y)\dd x\dd y\,,
\end{equation}
where $\psi$ is the stream function defined by
\begin{equation}\label{psidef}
  \psi(x) \,=\, \frac{1}{2\pi} \int_{\R^2} \log|x-y|\,\omega(y)\dd y\,, \qquad 
  x \in \R\,.
\end{equation}
This is an analogue of the function $f$ in the finite-dimensional example. Note
that the usual kinetic energy defined by $\frac 12\irt|u|^2\dd x$, where
$u = \nabla^\perp\psi$, is infinite for $\om\in\cP$. However, both definitions
of the energy coincide when $\irt\om\dd x=0$, which is the case for instance if
$\omega$ is the difference of two vorticities in $\cP$ with the same mass. It
is also worth observing that the functional $E$ is not invariant under the
scaling transformation
$\om(x) \mapsto \om^{(\lambda)}(x) := \lambda^2\om(\lambda x)$ when
$M_0 = \irt\om\dd x\ne 0$. In fact, one can easily check that
\[
  E(\om^{(\lambda)}) \,=\, E(\om) + \frac {M_0^2}{4\pi}\,\log \lambda\,,
  \qquad \hbox{for all } \lambda > 0\,.
\]

\item[(v)]
The {\em conserved quantities} induced by Euclidean symmetries. These 
are the first order moments $M_1, M_2$ and the symmetric second order 
moment $I$ defined by 
\begin{equation}\label{MjIdef}
  M_j(\om) \,=\, \irt x_j\om(x)\dd x\,, \qquad j=1,2\,, \qquad
  I(\om) \,=\, \irt |x|^2\om(x)\dd x\,.
\end{equation}
Note that $M_1, M_2$ are associated to the translational symmetry, 
via Noether's theorem, and $I$ to the rotational symmetry.  
\end{itemize}

With these definitions, the Euler equation can be written in the form
$\partial_t \om = \bm{\{}E(\om),\om\bm{\}}$, where $\bm{\{}\,\cdot\,,\cdot\,
\bm{\}}$ denotes the Poisson bracket on $\cP$, see Section~\ref{ssecA5}.
Any steady state $\bar\om \in \cP$ is a critical point of the Hamiltonian $E$ on
the orbit $\cO_{\bar\om}$. Stability can be inferred when the restriction of the
energy $E$ to $\cO_{\bar\om}$ has a strict local extremum at $\bar\om$. In what
follows, we focus on the maximizers of the energy, which correspond to radially
symmetric vortices.

\subsection{The constrained maximization of the energy in $\cP$}
\label{ssec13}

Under our assumptions, it is easy to determine the maximizers of the Hamiltonian
$E$ under the constraints given by the functions $h(a,\om)$ for
$a\in(0,\infty)$. Indeed, for any $\om\in\cP$, the orbit $\cO_\om$ contains a
unique element $\om^*$ that is radially symmetric and nonincreasing in the
radial direction; this is the {\em symmetric decreasing rearrangement} of $\om$
\cite{LL}. The Riesz's rearrangement inequality then shows that $E(\om) \le 
E(\om^*)$ for all $\om \in \cO_{\om^*}$, with equality if and only if $\om$ is a 
translate of $\om^*$, see \cite[Lemma~2]{Carlen-Loss}. Of course $\om^*$
is a stationary solution of the Euler equation, which represents a radially
symmetric vortex with nonincreasing vorticity profile. Our main focus here will
be on the analogue of the quadratic form~\eqref{I5} for the steady state
$\bar\om = \om^*$.

First, the analogue of the Lagrange function~\eqref{I9} is
\[
  E(\om)-\int_0^\infty \Lambda(a)h(a,\om)\dd a \,=\, E(\om)-\int_0^\infty\Lambda(a)
  \left(\irt\chi\bigl(\om(x)-a\bigr)\dd x\right)\dd a\,,
\]
where the quantities $\Lambda(a)$ for $a\in(0,\infty)$ can be thought of as the 
Lagrange multipliers. The role of the discrete index $j$ in~\eqref{I9} is now played 
by the continuous parameter $a > 0$. Defining\footnote{The reason for the minus sign
  in~\eqref{ni5} will become clear later.}
\begin{equation}\label{ni5}
  \Phi(s) \,=\, -\int_0^\infty \Lambda(a)\chi(s-a)\dd a \,=\, -\int_0^s\Lambda(a)\dd a\,,
   \qquad s > 0\,,
\end{equation}
we see that the Lagrange function can also be expressed as
\begin{equation}\label{cEdef}
  F(\om) \,=\, E(\om)+\irt\Phi(\om(x))\dd x\,, \qquad \om \in \cP\,.
\end{equation}
This quantity will be referred to later as the ``free energy'' of the vorticity
distribution $\om$, a terminology that will be discussed in Section~\ref{ssec14}
below. 

Next, the analogue of the stationarity condition~\eqref{I4} at $\bar \om = \om^*$ is
$F'(\bar\omega) = 0$, where the linear form $\eta \mapsto F'(\bar\omega)\eta$ is 
defined for all $\eta \in T_{\bar\om}\cP$ by 
\[
  F'(\bar\om)\eta \,=\, \irt \Bigl(-\bar\psi(x) + \Phi'(\bar\om(x))\Bigr)\eta(x)\dd x\,,
  \qquad \bar\psi(x) \,=\, \frac1{2\pi}\irt\log|x-y|\,\bar\om(y)\dd y\,.
\]
Stationarity is thus equivalent to the relation $\bar\psi(x) = \Phi'(\bar\om(x))$ 
for all $x \in \R^2$. Finally the analogue of \eqref{I5} is the quadratic form 
$\eta \mapsto F''(\bar\om)[\eta,\eta]$, where
\[
  F''(\bar\om)[\eta,\eta] \,=\, \irt\left(-\vf\eta + \Phi''(\bar\om)\eta^2\right)\dd x\,,
  \qquad \vf(x) \,=\, \irt\frac1{2\pi}\log|x-y|\,\eta(y)\dd y\,.
\]
Using the relation $\nabla\bar\psi(x)=\Phi''(\bar\om(x))\nabla\bar\om(x)$, the second
variation can be rewritten in the form
\begin{equation}\label{ni9}
  F''(\bar\om)[\eta,\eta] \,=\, \irt\Bigl(-\vf\eta+\frac{\nabla\bar\psi}{\nabla\bar \om}
  \,\eta^2\Bigr)\dd x \,=\, 2E(\eta) + \irt \frac{\nabla\bar\psi}{\nabla\bar \om}
  \,\eta^2 \dd x\,,
\end{equation}
which is well-known from Arnold's work. Note that the ratio $\frac{\nabla\bar\psi}{
\nabla\bar \om}$ is meaningful only when the vector $\nabla\bar\om(x)$ is nonzero and 
collinear with $\nabla\bar\psi(x)$ for almost all $x\in\R^2$. This condition is 
obviously satisfied for all radially symmetric vortices with strictly decreasing 
vorticity profile.

\subsection{Overview of our results}\label{ssec14}

We are now able to describe more precisely the results of this paper. We
consider a general family of radially symmetric vortices $\bar\om \in \cP$ with
vorticity profile satisfying Hypotheses~\ref{omegahyp} below. Typical examples
are the ``algebraic vortex'' $\bar\om(x) = (1+|x|^2)^{-\kappa}$, where $\kappa > 1$
is a parameter, and the Oseen vortex for which $\bar\om(x) = e^{-|x|^2/4}$.

\subsubsection{Arnold's quadratic forms with and without constraints}

In Section~\ref{sec2}, we study in detail the quadratic form \eqref{ni9}
associated with the second variation of the Lagrange function \eqref{cEdef} at
the steady state $\bar\om \in \cP$, paying some attention to the
functional-analytic questions. First of all, while we know from the constrained
maximization result that the restriction of that form to the tangent space
$T_{\bar\om}\cO_{\bar\om}$ is negative, it is not clear if this restriction is
strictly negative definite, and if so in which function space. Our first main
result is Theorem~\ref{Jprop1}, where we show that, if two neutral
directions corresponding to translational symmetry are disregarded, the
restriction to $T_{\bar\om}\cO_{\bar\om}$ of the quadratic form \eqref{ni9} is
indeed strictly negative in an appropriate weighted $L^2$ space. The proof
ultimately relies on a variant of the Krein-Rutman theorem. 

We next investigate the index of the quadratic form \eqref{ni9} on a much larger
subspace, corresponding to perturbations $\eta \in T_{\bar\om}\cP$ satisfying
$\irt \eta(x)\dd x = 0$. In other words, we relax all constraints given by the
Casimir functions \eqref{ni1}, except for the mass $M_0$ defined in \eqref{M0},
which is still supposed to be constant. A priori there is no reason why the form
\eqref{ni9} should be negative definite in this larger sense, and indeed
Theorem~\ref{Jprop2} shows that this is not always the case. More precisely, we
show that negativity holds in the large sense if and only if the optimal
constant in some weighted Hardy inequality (where the weight function depends on
the vorticity profile $\bar\omega$) is smaller than $1$. While that condition is
not easy to check in general, we deduce from Corollary~\ref{Hardycor} that it is
fulfilled at least for the Oseen vortex, as well as for the algebraic vortex
$\bar\om(x) = (1+|x|^2)^{-\kappa}$ if $\kappa \ge 2$.

Although the mass constraint is rather natural, one may wonder if, for some
vorticity profiles, the quadratic form \eqref{ni9} can be negative definite for
all perturbations $\eta \in T_{\bar\om}\cP$; this question is briefly discussed
in Section~\ref{ssec23}. Finally, in Section~\ref{ssec24}, we give a fairly
explicit expression of the energy $E(\bar\om)$ in terms of the constraints
$h(a,\bar\om)$ for all $a > 0$, see Proposition~\ref{en-h}. One obtains in this
way an infinite-dimensional analogue of the quantity $M(c_1,\dots,c_n)$ defined
in \eqref{Mf}. Among other things, we justify our claim that the index of the
quadratic form \eqref{I5} is related to the concavity of the function \eqref{Mf}
(which is not new, of course), and we discuss a similar link in the
infinite-dimensional case.

As an aside, we mention that the stability of radially symmetric vortices for
the 2d Euler equations can also be studied using other conserved quantities,
such as the second order symmetric moment $I$ defined in \eqref{MjIdef}, see
e.g. \cite[Chapter~3]{MP}.

\subsubsection{The global maximizers of the free energy}

Let $\bar\psi$ be the stream function associated with the radially symmetric
vortex $\bar\om$. We have seen that the analogue of the Lagrange function is
given by the ``free energy'' \eqref{cEdef}, where the function $\Phi$ is defined, 
up to an additive constant, by the relation $\bar\psi(x) = \Phi'(\bar\om(x))$. 
The appellation ``free energy'' is partially justified by a (loose) analogy of
formula~\eqref{cEdef} with the classical thermodynamical expression for the
free energy
\begin{equation}\label{fe2}
  F \,=\, U-TS\,.
\end{equation}
Here $U$ is the internal energy (of a suitable system), $T$ is the temperature,
and $S$ is the entropy. In \eqref{cEdef}, the energy $E$ is analogous to $U$,
the integral $\int_{\R^2} \Phi(\om(x))\dd x$ is analogous to $S$, and one can
argue that it is reasonable to take $T = -1$. Of course, $T$ has nothing to do
with the real temperature of the fluid, but it should roughly be thought of as
the statistical mechanics temperature of our system in the sense of Onsager
\cite{Onsager}. We have not attempted to make this connexion rigorous, which
would take us in a different direction. 

In Section~\ref{sec3}, we consider vortices $\bar\om$ which are {\em global
  maximizers} of the free energy $F(\om)$ for all $\om \in \cP$ satisfying
$\irt \om\dd x = \irt\bar\om\dd x$. Such equilibria can be expected to have
strong stability properties, and may be useful for other purposes too.  Using a
direct approach, in the sense of the calculus of variations, we prove the
existence of global maximizers under fairly general assumptions on the function
$\Phi$, see Theorem~\ref{thm:globmax}. However, we do not have any efficient
method to determine if a given vortex $\bar\om$ is a global maximizer or not. A
necessary condition is of course that the quadratic form \eqref{ni9} be negative
on perturbations $\eta$ with zero mean, see Theorem~\ref{Jprop2}, but there is
no reason to believe that this is sufficient. Numerical evidence indicates that
the Oseen vortex is a global maximizer, and so are the algebraic vortices
$\bar\om(x) = (1+|x|^2)^{-\kappa}$ for $\kappa \ge 2$. In the particular case
$\kappa = 2$, maximality can be deduced from the logarithmic
Hardy-Littlewood-Sobolev inequality
\begin{equation}\label{onofri}
  \int_{\R^2}\int_{\R^2}\log\frac{1}{|x-y|}\,\om(x)\om(y)\dd x\dd y \,\le\, 
  \frac12\int_{\R^2} \om(x)\log(\om(x)) + \frac{1+\log(\pi)}{2}\,,
\end{equation}
which holds for all $\omega \in \cP$ with $M_0(\omega) = 1$, see \cite{Carlen-Loss}.
We mention that \eqref{onofri} is related to Onofri's sharp version of the
Moser-Trudinger inequality \cite{Onofri}. 

\subsubsection{The effect of viscosity --- application to Oseen vortices}

In Section~\ref{sec4}, we consider the stability of the Gaussian
vortex under the evolution defined by the Navier-Stokes equation $\partial_t
\omega + u\cdot\nabla \omega = \nu\Delta \omega$, where $\nu > 0$ is the
viscosity parameter. More precisely, we show that the quadratic form
\eqref{ni9} can be used to give an alternative proof of the local stability
results established in \cite{GW2}. We believe that a proof relying on the
second variation of the energy is of some interest, because the analogue of
the form \eqref{ni9} can be defined for more complicated vortex structures as
well, whereas the simpler approach in \cite{GW2} may be more difficult to
adapt.

The addition of the viscous term results in important new issues: the radial
vortices are no longer steady states and the orbits \eqref{ni2} are no longer
invariant under the evolution, so that much of the geometric picture underlying
the Euler equation is destroyed. The first problem is settled by introducing
self-similar variables and restricting ourselves to Oseen's vortex, which is a
stationary solution of the Navier-Stokes equation in these new coordinates.
Thanks to Theorem~\ref{Jprop2} and Corollary~\ref{Hardycor}, the quadratic form
\eqref{ni9} is positive definite for all perturbations with zero mean. 
This form is invariant under the evolution defined by the linearized Euler equation
at the vortex, but not under the Navier-Stokes evolution due to the viscous term and
the nonlinearity. The effect of viscosity is measured by a second quadratic
form, which happens to have a favorable sign, see Theorem~\ref{Qprop}. We do not
know if this is just a lucky coincidence, or if there are deeper reasons behind
that. In any event, this nice structure allows us to recover the local
stability result of \cite{GW2}, except for a slight difference in the choice
of the function space, see Theorem~\ref{Oseenprop}. Again, we emphasize that
the functional setting used in \cite{GW2} relies in an essential way on the radial
symmetry of Oseen's vortex, through the existence of conserved quantities such
as the moment $I$ in \eqref{MjIdef}, whereas our new approach can, at least in
principle, be adapted to more general situations, where other methods do not work.

\medskip\noindent{\bf Acknowledgments.} ThG is partially supported by the grant 
SingFlows ANR-18-CE40-0027 of the French National Research Agency (ANR). The 
research of VS is supported in part by grant DMS 1956092 from the National
Science Foundation.

\section{The second variation of the energy}\label{sec2}

In this section, we study the coercivity on various subspaces of the quadratic
form \eqref{ni9} which represents the second variation of the free energy
\eqref{cEdef} at a radially symmetric vortex $\bar\omega \in \cP$. We assume
that $\bar\omega(x) = \omega_*(|x|)$ for all $x \in \R^2$, and that the
vorticity profile $\omega_* : [0,+\infty) \to \R$ is a $C^2$ function with the
following properties\:

\begin{hyp}\label{omegahyp} The vorticity profile $\omega_* \in
C^2([0,+\infty))$ satisfies\\[1mm]
1) $\omega_*(0) > 0$, $\omega_*'(0) = 0$, and $\omega_*''(0) < 0$; \\[1mm]
2) $\omega_*'(r) < 0$ for all $r > 0$, and $\omega_*(r) \to 0$ as
$r \to +\infty$; \\[1mm]
3) there exists $C > 0$ and $\beta > 2$ such that $|\omega_*'(r)| \le C(1+r)^{-\beta-1}$ 
for all $r > 0$. 
\end{hyp}

\noindent
It follows in particular from 2), 3) that $\omega_*(r) = -\int_r^\infty 
\omega_*'(s)\dd s$, so that
\begin{equation}\label{om*prop}
  0 \,<\, \omega_*(r) \,\le\, \frac{C}{(1+r)^\beta} \quad
  \forall\, r > 0\,, \qquad \hbox{and}\quad
  0 \,<\, \int_0^\infty r\omega_*(r)\dd r \,<\, \infty\,. 
\end{equation}
Let $\bar\psi$ be the stream function associated with $\bar\omega$ as in
\eqref{psidef}. We have $\bar\psi(x) = \psi_*(|x|)$, where the stream profile
$\psi_* : [0,+\infty) \to \R$ satisfies
\begin{equation}\label{psi*def}
  \psi_*''(r) + \frac{1}{r}\,\psi_*'(r) \,=\, \omega_*(r)\,, \qquad 
  \hbox{hence}\quad  \psi_*'(r) \,=\, \frac{1}{r} \int_0^r s\omega_*(s) 
  \dd s\,, \quad \forall\, r > 0\,. 
\end{equation}
We introduce the weight function $A : [0,+\infty) \to \R$ defined by 
$A(0) = -\omega_*(0)/(2\omega_*''(0))$ and 
\begin{equation}\label{Adef}
  A(r) \,=\, -\frac{\psi_*'(r)}{\omega_*'(r)} \,=\, -\frac{1}{r\omega_*'(r)} 
  \int_0^r s\omega_*(s) \dd s\,, \quad r > 0\,. 
\end{equation}
Hypotheses~\ref{omegahyp} ensure that $A \in C^0([0,+\infty)) \cap C^1((0, +\infty))$.
Moreover, there exists a constant $C > 0$ such that $A(r) \ge C(1+r)^\beta$ for
all $r \ge 0$.

Let $\cA : \R^2 \to (0,\infty)$ be the radially symmetric extension of $A$ to
$\R^2$, namely $\cA(x) = A(|x|)$ for all $x \in \R^2$. We introduce the weighted 
$L^2$ space $X$ defined by
\begin{equation}\label{Xdef}
  X \,=\, \Bigl\{\omega \in L^2(\R^2)~;\, \|\omega\|_X^2 := 
  \int_{\R^2} \cA(x)|\omega(x)|^2 \dd x < \infty\Bigr\}\,,
\end{equation}
so that $\omega \in X$ if and only if $\cA^{1/2}\omega \in L^2(\R^2)$. Our assumptions
ensure that $\cA^{-1} \in L^1(\R^2)$, and using H\"older's inequality we easily
deduce that $X \hookrightarrow L^1(\R^2)$. We also consider the closed subspaces 
$X_1 \subset X_0 \subset X$ defined by 
\begin{equation}\label{X0def}
\begin{split}
  X_0 \,&=\, \Bigl\{\omega \in X~;\, \int_{\R^2} \omega(x) \dd x = 0\Bigr\}\,, \\
  X_1 \,&=\, \Bigl\{\omega \in X_0~;\, \int_{\R^2} \frac{x_j}{|x|}\,\omega(x) \dd x = 0
  \,\hbox{ for }\, j = 1,2\Bigr\}\,.
\end{split}
\end{equation}

We observe that, for any $\omega \in X$, the energy $E(\omega)$ introduced
in \eqref{Edef} is well defined. This a consequence of the following 
classical estimate, whose proof is reproduced in Section~\ref{ssecA1}
for the reader's convenience. 

\begin{prop}\label{prop:Edef}
Assume that $\omega \in L^1(\R^2)$ satisfies
\begin{equation}\label{omcond}
  \int_{\R^2} |\omega(x)|\,\log(1+|x|)\dd x \,<\, \infty\,, \quad \hbox{and}
  \quad \int_{\R^2} |\omega(x)|\,\log\bigl(1+|\omega(x)|\bigr)\dd x \,<\, \infty\,. 
\end{equation}
Then the last member in \eqref{Edef} is well defined, and the energy $E(\omega)$ 
satisfies the bound
\begin{equation}\label{EEbound}
  |E(\omega)| \,\le\, C \|\omega\|_{L^1} \biggl(\int_{\R^2} |\omega(x)|\,\log(2+|x|)
  \dd x + \int_{\R^2} |\omega(x)|\,\log_+\frac{|\omega(x)|}{\|\omega\|_{L^1}}
  \dd x\biggr)\,,
\end{equation}
where $\log_+(a) = \max\bigl(\log(a),0\bigr)$. If moreover $\int_{\R^2}\omega(x)
\dd x = 0$, then $E(\omega) = \frac12 \int_{\R^2}|u|^2\dd x$ where
\begin{equation}\label{BS}
  u(x) \,=\, \nabla^{\perp}\psi(x) \,=\, \frac{1}{2\pi}\int_{\R^2} \frac{(x-y)^\perp}{
  |x-y|^2}\,\omega(y)\dd y\,, \qquad x \in \R^2\,.
\end{equation}
 \end{prop}

Since any $\omega \in X$ obviously satisfies \eqref{omcond}, we can consider 
the quadratic form $J$ on $X$ defined by $J(\omega) = \frac12\,\|\omega\|_X^2 
- E(\omega)$, or explicitly
\begin{equation}\label{JdefX}
  J(\omega) \,=\, \frac12 \int_{\R^2}\cA(x)\om(x)^2 \dd x \,+\, \frac1{4\pi}\int_{\R^2}
  \int_{\R^2}\log|x-y|\,\om(x)\om(y)\dd x\dd y\,, \quad \omega \in X\,.
\end{equation}
In the particular case where $\omega \in X_0$, namely when $\omega$ has zero 
average over $\R^2$, Proposition~\ref{prop:Edef} gives the alternative expression
\begin{equation}\label{Jdef}
  J(\omega) \,=\, \frac{1}{2}\int_{\R^2} \Bigl(\cA(x)\omega(x)^2 - 
  |u(x)|^2\Bigr)\dd x\,, \qquad \omega \in X_0\,,
\end{equation}
where $u$ is the velocity field associated with $\omega$ via the Biot-Savart 
formula \eqref{BS}. In view of \eqref{ni9} and \eqref{Adef}, we have 
$J = -\frac12F''(\bar\omega)$, where $F''(\bar\omega)$ is the second 
variation of the free energy \eqref{cEdef} at the equilibrium $\bar\omega$. 
It is clear that $X$ is the largest function space on which this second
variation is well defined. 

Our main goal in this section is to study the positivity and coercivity 
properties of the quadratic form $J$ on the spaces $X$, $X_0$, and $X_1$ 
defined in \eqref{Xdef}, \eqref{X0def}. To formulate our results, it is useful 
to decompose $X = X_\rs \oplus X_\rs^\perp$, where 
\begin{equation}\label{Xrsdef}
   X_\rs \,=\, \bigl\{\omega \in X\,;\, \omega \hbox{ is radially symmetric}
  \bigr\}\,, 
\end{equation}
and $X_\rs^\perp$ is the orthogonal complement of $X_\rs$ in the Hilbert space
$X$. Referring to the geometric picture of Section~\ref{ssec12}, we consider
$X_\rs^\perp$ as the tangent space to the orbit $\cO_{\bar\omega}$ at
$\bar\omega$.  This interpretation can be formally justified as follows: if
$\bar\omega \in X$ is smooth, the tangent space $T_{\bar\omega}\cO_{\bar\omega}$
is spanned by vorticities of the form $v\cdot\nabla \bar\omega$, where $v$ is a
(smooth and localized) divergence-free vector field, and using polar coordinates
as in Section~\ref{ssec21} below one verifies that such vorticities are indeed
orthogonal in $X$ to all radially symmetric functions. A contrario, since 
there is a one-to-one correspondence in $\cP$ between orbits and symmetric
decreasing rearrangements, it is clear that any radially symmetric perturbation
of the equilibrium $\bar\omega$ is transverse to the orbit $\cO_{\bar\omega}$.

It is easy to verify that $J(\omega_1 + \omega_2) = J(\omega_1) + J(\omega_2)$ 
when $\omega_1 \in X_\rs$ and $\omega_2 \in X_\rs^\perp$, so that the restrictions 
of $J$ to $X_\rs$ and $X_\rs^\perp$ can be studied separately. We first 
consider the tangent space $X_\rs^\perp$ in Section~\ref{ssec21}, and 
postpone the study of radially symmetric perturbations (with zero or 
nonzero mass) to Sections~\ref{ssec22} and \ref{ssec23}. 

\begin{rem}\label{psiprem}
Differentiating the first equality in \eqref{psi*def}, we see that the 
function $\phi = \psi_*'$ satisfies
\begin{equation}\label{psipeq}
  \bigl(L_0 \phi\bigr)(r) \,:=\, -\phi''(r) - \frac{1}{r}\,\phi'(r) + \frac{1}{r^2}
  \,\phi(r) \,=\, \frac{1}{A(r)}\,\phi(r)\,, \qquad r > 0\,,
\end{equation}
where $A(r) \ge C(1+r)^\beta$. Since $\phi > 0$, Sturm-Liouville theory asserts
that $\mu = 1$ is the lowest eigenvalue of the (generalized) eigenvalue problem
$L_0 \phi = \mu A^{-1} \phi$ on $\R_+$, with boundary conditions $\phi(0)
= \phi(+\infty) = 0$, see \cite{CL,Ha}.  This observation will be used later.
\end{rem}

\begin{rem}\label{Rem:Hyp}
Hypotheses~\ref{omegahyp} are sufficient for our results to hold, but can be 
relaxed in several ways. In particular, we can consider vortices that are not 
smooth at the origin, but the assumption that $\omega_*'(r) < 0$ for all 
$r > 0$ seems essential. This excludes vortices with compact support from our 
considerations, but as our motivation comes from applications to the Navier-Stokes
equations, we did not try to optimize our set of assumptions.
\end{rem}

\subsection{Positivity of the quadratic form $J$ on $X_\rs^\perp$}
\label{ssec21}

\begin{thm}\label{Jprop1}
Under Hypotheses~\ref{omegahyp}, the quadratic form $J$ defined by \eqref{Jdef} 
is nonnegative on the space $X_\rs^\perp \subset X_0$. Moreover, there exists 
a constant $\gamma > 0$ such that
\begin{equation}\label{Jcoercive}
  J(\omega) \,\ge\, \frac{\gamma}{2}\int_{\R^2} \cA(x)\omega(x)^2\dd x\,,
  \qquad \hbox{for all }\omega \in X_\rs^\perp \cap X_1\,.
\end{equation}
\end{thm}

\begin{proof}
We introduce polar coordinates $(r,\theta)$ in $\R^2$, and given any 
$\omega \in X_\rs^\perp$ we use the Fourier decomposition
\begin{equation}\label{omFourier}
  \omega\bigl(r\cos(\theta),r\sin(\theta)\bigr) \,=\, \sum_{k\neq 0}
  \omega_k(r)\,e^{ik\theta}\,, \qquad r > 0\,,\quad\theta \in \R/(2\pi\Z)\,,
\end{equation}
where the sum runs over all nonzero integers $k \in \Z \setminus \{0\}$. By
Parseval's relation we have
\begin{equation}\label{Jk}
\begin{split}
  \int_{\R^2} \cA(x)\omega(x)^2\dd x \,&=\, 2\pi \sum_{k\neq 0} 
   \int_0^\infty A(r)\,|\omega_k(r)|^2\,r\dd r\,, \\
  \int_{\R^2} |u(x)|^2\dd x \,=\, \int_{\R^2} \bigl(-\Delta^{-1}\omega\bigr)(x) 
  \,\omega(x)\dd x \,&=\, 2\pi \sum_{k\neq 0}
  \int_0^\infty B_k[\omega_k](r)\,\overline{\omega_k}(r)\,r\dd r\,,
\end{split}
\end{equation}
where $B_k$ is the integral operator on the half-line $\R_+$ defined by
the formula
\begin{equation}\label{Bkdef}
  \bigl(B_k[f]\bigr)(r) \,=\, \frac{1}{2|k|}\int_0^\infty \min\Bigl(
  \frac{r}{s},\frac{s}{r}\Bigr)^{|k|} f(s)\,s\dd s\,,   \qquad r > 0\,.
\end{equation}
Note that $g = B_k[f]$ is the unique solution of the ODE
\begin{equation}\label{gkdiff}
  -g''(r) - \frac{1}{r}\,g'(r) + \frac{k^2}{r^2}\,g(r) \,=\, f(r)\,,
  \qquad r > 0\,,
\end{equation}
which is regular at the origin and converges to zero at infinity. 

In view of \eqref{Jk}, the proof of Theorem~\ref{Jprop1} reduces
to the study of the one-dimensional inequality
\begin{equation}\label{ineqk}
  \int_0^\infty \bigl(B_k[f]\bigr)(r)\,\overline{f}(r)\,r\dd r \,\le\, C_k 
  \int_0^\infty A(r) |f(r)|^2 \,r\dd r\,,
\end{equation}
which depends on the angular Fourier parameter $k \in \Z\setminus\{0\}$.  More
precisely, the quadratic form $J$ is nonnegative on $X_\rs^\perp$ if and only
if, for all $k \neq 0$, inequality \eqref{ineqk} holds with some constant
$C_k \le 1$. In addition, we have the lower bound \eqref{Jcoercive} on the
subspace $X_\rs^\perp \cap X_1$ if and only if inequality \eqref{ineqk} holds
with a better constant $C_k \le 1-\gamma$ for all $k \neq 0$, assuming when
$|k| = 1$ that $f$ satisfies the additional condition
\begin{equation}\label{ortho1}
  \int_0^\infty f(r)\,r\dd r \,=\, 0\,.
\end{equation}

\smallskip
It remains to establish \eqref{ineqk} for all $k \in \Z\setminus\{0\}$. An 
inspection of the explicit formula \eqref{Bkdef} leads to the following simple 
observations\:\\[1mm] 
\null\hspace{8pt}1) The operator $B_k$ preserves positivity\: if $f \ge 0$, 
then $B_k[f] \ge 0$;\\[1mm]
\null\hspace{8pt}2) The following pointwise estimate holds\:
$|B_k[f]| \le B_k[|f|]$;\\[1mm]
\null\hspace{8pt}3) If $f \ge 0$ then $B_k[f] \le |k|^{-1} B_1[f]$.  

\medskip\noindent As a consequence, to show that $J$ is nonnegative on
$X_\rs^\perp$, it is sufficient to prove inequality \eqref{ineqk}
in the particular case where $|k| = 1$ and $f \ge 0$. Setting $h = A^{1/2} f$,
we write that inequality in the equivalent form
\begin{equation}\label{ineq1}
  \int_0^\infty \bigl(\wtB_1[h]\bigr)(r)\,h(r)\,r\dd r \,\le\, 
  C_1 \int_0^\infty h(r)^2 \,r\dd r\,,
\end{equation}
where $\wtB_1[h] = A^{-1/2} B_1[A^{-1/2}h]$. The following assertions play 
a crucial role in our argument\: 

\smallskip\noindent{\bf Claim 1\:} The operator $\wtB_1$ is {\em 
selfadjoint and compact} in the (real) space $Y = L^2(\R_+,r\dd r)$.\\
Indeed, take $h \in Y$ with $\|h\|_Y \le 1$, and denote $f = A^{-1/2}h$, 
$g = B_1[f] = A^{1/2}\wtB_1[h]$. Applying \eqref{Bkdef} with $|k| = 1$,
we see that
\[
  g(r) \,=\, \frac{1}{2r}\int_0^r A(s)^{-1/2}h(s)\,s^2\dd s \,+\, \frac{r}{2} 
  \int_r^\infty A(s)^{-1/2} h(s)\dd s\,, \qquad r > 0\,,
\]
and using H\"older's inequality we deduce
\begin{equation}\label{gbd1}
  |g(r)| \,\le\, \biggl\{\frac{1}{2r}\Bigl(\int_0^r A(s)^{-1}\,s^3\dd s
  \Bigr)^{1/2} +\, \frac{r}{2}\Bigl(\int_r^\infty A(s)^{-1}\,s^{-1}\dd s
  \Bigr)^{1/2}\biggr\}\,\|h\|_Y\,.
\end{equation}
As $A(r) \ge C(1+r)^\beta$, the right-hand side of \eqref{gbd1} is uniformly bounded,
so that $\|g\|_{L^\infty} \le C$ for some universal constant $C$.  On the other
hand, since $g$ satisfies the ODE \eqref{gkdiff} with $k = 1$ and
$f = A^{-1/2}h$, a standard energy estimate yields the bound
\begin{equation}\label{gbd2}
  \int_0^\infty \Bigl(g'(r)^2 + \frac{g(r)^2}{r^2}\Bigr)r\dd r
  \,=\, \int_0^\infty g(r) A(r)^{-1/2} h(r)\,r\dd r \,\le\, 
  \|g\|_{L^\infty} \|A^{-1/2}\|_Y \|h\|_Y \,\le\, C\,.
\end{equation}
In view of \eqref{gbd1} and \eqref{gbd2}, the Fr\'echet-Kolmogorov theorem
\cite[Thm~XIII.66]{RS4} implies that the function $\wtB_1[h] = A^{-1/2}g$ lies
in a compact set of $Y$, so that the operator $\wtB_1$ is compact. To prove that
$\wtB_1$ is selfadjoint, we take $h_1, h_2 \in Y$ and observe that
\[
  \int_0^\infty \bigl(\wtB_1[h_1]\bigr)(r)\,h_2(r)\,r\dd r \,=\, 
  \int_0^\infty \Bigl(g_1'(r)g_2'(r) + \frac{g_1(r)g_2(r)}{r^2}\Bigr)r\dd r\,,
\]
where $g_j = B_1[A^{-1/2}h_j]$ for $j = 1,2$. This expression is clearly 
a symmetric function of $(h_1,h_2)$. 

\smallskip\noindent{\bf Claim 2\:} The {\em spectral radius} of $\wtB_1$ is
equal to $1$, and $\lambda = 1$ is a {\em simple eigenvalue} of $\wtB_1$.\\
To see that, we first observe that $\lambda = 1$ is an eigenvalue of
$\wtB_1$ with a positive eigenfunction. Indeed, using \eqref{psi*def},
it is straightforward to verify that the function $g = \psi_*'$
satisfies the ODE \eqref{gkdiff} with $k = 1$ and $f =
-\omega_*'$. This shows that $B_1[-\omega_*'] = \psi_*'$, hence defining
$h = A^{-1/2}\psi_*' = -A^{1/2}\omega_*'$ we conclude that $\wtB_1[h] = h$.
On the other hand, assume that $\lambda > 0$ is an eigenvalue of $\wtB_1$, 
with eigenfunction $h \in Y$. Defining $f = A^{-1/2}h$, we see that 
$B_1[f] = \lambda Af$, so that the function $g = B_1[f]$ satisfies the
generalized eigenvalue problem
\begin{equation}\label{gkdiff2}
  -g''(r) - \frac{1}{r}\,g'(r) + \frac{1}{r^2}\,g(r) \,=\, 
  \mu\,\frac{g(r)}{A(r)}\,,\qquad r > 0\,,
\end{equation}
with the boundary conditions $g(0) = g(+\infty) = 0$, where $\mu = 1/\lambda$,
We already observed that $\mu = 1$ is the lowest eigenvalue of \eqref{gkdiff2},
see Remark~\ref{psiprem}.  It follows that $\lambda = 1$ is the largest
eigenvalue of the integral operator $\wtB_1$, whose spectral radius is therefore
equal to $1$. The argument above also shows that all positive eigenvalues of
$\wtB_1$ are simple, because \eqref{gkdiff2} is a second-order differential
equation.

\smallskip It is now a simple task to conclude the proof of
Theorem~\ref{Jprop1}.  Claims 1 and 2 imply the validity of inequality
\eqref{ineq1} with $C_1 = 1$. We deduce that \eqref{ineqk} holds for $|k| = 1$
with $C_k = 1$, and (invoking observation 3 above) for $|k| \ge 2$ with
$C_k \le 1/|k|$. This shows that the quadratic form $J$ is nonnegative on
$X_\rs^\perp$. On the other hand, if we assume that $\omega \in X_\rs^\perp\cap X_1$, 
the function $f = \omega_{\pm 1}$ satisfies condition \eqref{ortho1}, which means 
that $h = A^{1/2}f$ is orthogonal in $Y$ to the one-dimensional subspace $Y_0$ 
spanned by the positive function $\chi = A^{-1/2}$. It is clear that $Y_0^\perp$ 
does not contain any positive function, and in particular does not include the 
principal eigenfunction $h_0 = -A^{1/2}\omega_*'$ of the operator $\wtB_1$. 
So, applying  Lemma~\ref{LemCoercive} and Remark~\ref{remCoercive} below, 
we deduce that $\1 - \wtB_1 > 0$ on $Y_0^\perp$, which means that inequality 
\eqref{ineq1} holds on $Y_0^\perp$ with some constant $C_1' < 1$. Taking 
into account the other values of $k$, for which $C_k \le 1/|k| \le 1/2$, we 
conclude that estimate \eqref{Jcoercive} holds with $\gamma = \min(1/2,1{-}C_1')$.
\end{proof}

\begin{rem}\label{KRrem}
The Krein-Rutman theorem \cite[Thm~19.2]{Dei} asserts that the spectral radius 
of the compact and positivity-preserving operator $\wtB_1$ is an eigenvalue 
with positive eigenfunction. However, since the cone of positive functions 
has empty interior in $Y$, we cannot apply Theorem~19.3 in \cite{Dei} to conclude 
that $\wtB_1$ has a {\em unique} eigenvalue with positive eigenfunction. For this 
reason, we prefer invoking Sturm-Liouville theory to prove that $1$ is the 
largest eigenvalue of $\wtB_1$. 
\end{rem}

\begin{rem}\label{X1rem}
If $\beta > 4$ in Hypotheses~\ref{omegahyp}, the conclusion of Theorem~\ref{Jprop1}
remains valid, with the same proof, if the subspace $X_1$ is replaced by
\begin{equation}\label{X2def}
  \cX_1 \,=\, \Bigl\{\omega \in X_0~;\, \int_{\R^2} x_j\,\omega(x) \dd x = 0
  \,\hbox{ for }\, j = 1,2\Bigr\}\,.
\end{equation}
This possibility will be used in Section~\ref{sec4}. 
\end{rem}

\subsection{Positivity of the quadratic form $J$ on $X_\rs \cap X_0$}
\label{ssec22}

The quadratic form $J$ is not necessarily positive when considered 
on the subspace $X_\rs \cap X_0$, which consists of radially symmetric 
functions with zero mean. This question is related to the optimal 
constant in the weighted Hardy inequality 
\begin{equation}\label{Hardy}
  \int_0^\infty f(r)^2 \,\frac{\D r}{r} \,\le\, C_H  \int_0^\infty A(r)f'(r)^2 
  \,\frac{\D r}{r}\,, 
\end{equation}
where $f : [0,+\infty) \to \R$ is an absolutely continuous function with $f(0)
= f(+\infty) = 0$. Weighted Hardy inequalities are extensively studied in the
literature, see e.g. \cite[Section~1.3.2]{Maz}. In particular, it is known
that \eqref{Hardy} holds for {\em some} constant $C_H > 0$ if and only if the
positive function $A$ satisfies
\begin{equation}\label{Hardycond}
  \limsup_{r \to 0} \,\log\frac{1}{r}\int_0^r \frac{s}{A(s)}\dd s \,<\, \infty\,, 
  \quad \hbox{and} \quad 
  \limsup_{r \to +\infty} \,\log(r)\int_r^\infty \frac{s}{A(s)}\dd s \,<\, \infty\,. 
\end{equation}
Both conditions in \eqref{Hardycond} are fulfilled in our case, since $A(r) 
\ge C(1+r)^\beta$ for some $\beta > 2$. 

\begin{thm}\label{Jprop2}
Under Hypotheses~\ref{omegahyp}, the quadratic form $J$ defined by \eqref{Jdef} 
is coercive on $X_\rs \cap X_0$ if and only if Hardy's inequality \eqref{Hardy} 
holds for some $C_H < 1$. In that case we have
\begin{equation}\label{Jcoercive2}
  J(\omega) \,\ge\, \frac{\gamma}{2}\int_{\R^2} \cA(x)\omega(x)^2\dd x\,,
  \qquad \hbox{for all }\omega \in X_\rs \cap X_0\,,
\end{equation}
where $\gamma = 1-C_H$. 
\end{thm}

\begin{proof}
Given $\omega \in X_\rs \cap X_0$, we write $\omega(x) = \omega_0(|x|)$ and we
consider the stream function $\psi_0$ defined (up to an irrelevant additive
constant) by
\[
  \psi_0'(r) \,=\, \frac{1}{r}\int_0^r s\omega_0(s)\dd s \,=\, 
  -\frac{1}{r}\int_r^\infty s\omega_0(s)\dd s\,, \qquad r > 0\,.
\]
Defining $f(r) = r \psi_0'(r)$, we see that $f$ is absolutely continuous on
$\R_+$ with $f(0) = f(+\infty) = 0$. Moreover we have $\omega_0(r) = f'(r)/r$
and $u_0(r) := \psi_0'(r) = f(r)/r$ by construction. Finally the assumption
that $\omega_0 \in X_\rs \cap X_0$ ensures that $A^{1/2}\omega_0$ and $u_0$
belong to the space $Y = L^2(\R_+,r\dd r)$. We thus have
\begin{equation}\label{Jcoercive3}
  J(\omega) \,=\, \pi \int_0^\infty \Bigl(A(r)\omega_0(r)^2 - u_0(r)^2
  \Bigr)\,r \dd r \,=\, \pi \int_0^\infty \Bigl(A(r)
  f'(r)^2 - f(r)^2\Bigr)\frac{\D r}{r}\,,
\end{equation}
and using \eqref{Hardy} we conclude that \eqref{Jcoercive2} holds with
$\gamma = 1-C_H$. This proves that the quadratic form $J$ is coercive on
$X_\rs \cap X_0$ if $C_H < 1$. Conversely, if \eqref{Jcoercive2} holds for
some $\gamma > 0$,  it follows from \eqref{Jcoercive3} that inequality
\eqref{Hardy} is valid with $C_H = 1 - \gamma$.
\end{proof}

As is well known, the optimal constant in Hardy's inequality \eqref{Hardy} is
related to the lowest eigenvalue of a selfadjoint operator. A convenient
way of seeing this is to apply the change of variables $r = e^x$,
$h(x) = f(e^x)$, $B(x) = e^{-2x}A(e^x)$, which transforms \eqref{Hardy}
into the equivalent inequality
\begin{equation}\label{Hardy1}
  \int_\R h(x)^2 \dd x \,\le\, C_H  \int_\R B(x) h'(x)^2\dd x\,. 
\end{equation}
The integral in the right-hand side of \eqref{Hardy1} defines a closed
quadratic form on the Hilbert space $H = L^2(\R)$, with dense domain
$D = \{h \in H\,;\, B^{1/2}h' \in H\}$. Let
\[
  \B : D(\B) \longrightarrow H\,, \qquad h \,\longmapsto\,
  -\partial_x(B(x)\partial_x h)
\]
be the selfadjoint operator in $H$ associated with the quadratic form
\eqref{Hardy1} by Friedrich's representation theorem \cite{Ka}. Since $B(x) > 0$
for all $x \in \R$ we know that $\B$ is positive, and using the fact that $x^2
B(x)^{-1} \to 0$ as $|x| \to \infty$ it is easy to verify that $\B$ has compact
resolvent in $H$, hence purely discrete spectrum. The optimal constant in
$C_H$ in \eqref{Hardy1} is precisely the inverse of the lowest eigenvalue of
$\B$\:
\begin{equation}\label{CHformula}
  C_H \,=\, \max\bigl\{\lambda^{-1}\,;\, \lambda \in \spec(\B)\bigr\}\,.
\end{equation}

By Sturm-Louville's theory, if $\mu = C_H^{-1}$ is the lowest eigenvalue
of $\B$, there exists a positive eigenfunction $h \in D(\B)$ such that
$\B h = \mu h$. Setting $h(x) = f(e^x)$, we see that $f$ is a positive
solution of the ODE
\begin{equation}\label{fODE}
  -\partial_r\biggl(\frac{A(r)}{r}\,\partial_r f(r)\biggr) \,=\, \mu
  \,\frac{f(r)}{r}\,, \qquad r > 0\,,
\end{equation}
satisfying the boundary conditions $f(0) = f(+\infty) = 0$. Moreover
$\int_0^\infty A(r)f'(r)^2\dd r/r < \infty$ by construction. It is not easy to
guess from \eqref{fODE} whether $\mu$ is smaller or larger than $1$, but under
additional assumptions on the vortex profile it is possible to make another
change of variables which puts \eqref{fODE} into a form that allows
for a comparison with \eqref{psipeq}.

\begin{lem}\label{gODElem} If the function $A$ in \eqref{Adef} satisfies 
\begin{equation}\label{Anewcond}
  A \,\in\, C^2([0,+\infty))\,, \qquad\hbox{and}\qquad 
  \sup_{r \ge 1}\,\biggl(\frac{A(r)}{r^2} + \frac{A'(r)^2}{r^2A(r)}\biggr)
  \int_r^\infty\frac{s}{A(s)}\dd s \,<\, \infty\,,
\end{equation}      
then the function $g : [0,+\infty) \to \R$ defined by $g(r) = A(r)^{1/2}
f(r)/r$ is a solution of the ODE
\begin{equation}\label{gODE}
  -g''(r) - \frac{1}{r}\,g'(r) + \frac{1}{r^2}\,g(r) + V(r)g(r)  
  \,=\, \frac{\mu}{A(r)}\,g(r)\,,
\end{equation}
with boundary conditions $g(0) = g(+\infty) = 0$, where
\begin{equation}\label{Vdef}
  V(r) \,=\, \chi''(r) - \frac{1}{r}\,\chi'(r) + \chi'(r)^2\,, \qquad
  \hbox{and}\qquad \chi(r) \,=\, \frac{1}{2}\,\log(A(r))\,.
\end{equation}
\end{lem}

\begin{proof}
Since $f$ satisfies \eqref{fODE}, a direct calculation shows that $g(r) :=
A(r)^{1/2}f(r)/r$ is a solution of \eqref{gODE}, where the potential $V$ is
defined by \eqref{Vdef}. As for the boundary conditions, we recall that
$\int_0^\infty A(r)f'(r)^2\dd r/r < \infty$, hence $\int_0^\infty |f'(r)|\dd r
< \infty$. As $f(r) = \int_0^r f'(s)\dd s$, we have
\[
  \frac{|f(r)|}{r} \,\le\, \frac{1}{r}\biggl(\int_0^r \frac{s}{A(s)}\dd s
  \biggr)^{1/2}\biggl(\int_0^r A(s)f'(s)^2\,\frac{\D s}{s}\biggr)^{1/2}
  ~\xrightarrow[r \to 0]{}~ 0\,,
\]
which shows that $g(r) \to 0$ as $r \to 0$. Similarly, since $f(r) = -\int_r^\infty
f'(s)\dd s$, we have
\[
 |g(r)| \,\le\, \frac{A(r)^{1/2}}{r}\biggl(\int_r^\infty \frac{s}{A(s)}
  \dd s\biggr)^{1/2}\biggl(\int_r^\infty A(s)f'(s)^2\,\frac{\D s}{s}\biggr)^{1/2}
  ~\xrightarrow[r \to +\infty]{}~ 0\,,
\]
thanks to \eqref{Anewcond}. This concludes the proof. 
\end{proof}

\begin{rem}\label{rem:gadd}
The same arguments show that $r^2 g'(r) \to 0$ as $r \to 0$ and
$g'(r) \to 0$ as $r \to +\infty$, at least along appropriate sequences. 
\end{rem}
  
Let $L$ be the differential operator defined by
\begin{equation}\label{Ldef} 
  L \,=\, L_0 + V \,=\, -\partial_r^2 -\frac{1}{r}\,\partial_r + \frac{1}{r^2}
  + V(r)\,,
\end{equation}
where $L_0$ was introduced in \eqref{psipeq}. We know from \eqref{gODE} that
$L g = \mu A^{-1}g$, where $\mu = C_H^{-1}$ and $g$ is the positive function
defined in Lemma~\ref{gODElem}. On the other hand, we observed in
Remark~\ref{psiprem} that $L_0 \phi = A^{-1}\phi$, where $\phi = \psi_*'$
is also a positive function vanishing at the origin and at infinity.
Using Sturm-Liouville's theory, we easily deduce the following useful
criterion: 

\begin{cor}\label{Hardycor}
Under assumptions \eqref{Anewcond}, if the function $V$ defined by \eqref{Vdef}
does not change sign, the optimal constant in Hardy's inequality \eqref{Hardy}
satisfies $C_H \le 1$ if $V \ge 0$, and $C_H \ge 1$ if $V \le 0$;
moreover $C_H = 1$ only if $V$ is identically zero. 
\end{cor}

\begin{proof}
With the notations above, we have $L_0 \phi - A^{-1}\phi = 0$ and
\begin{equation}\label{STid}
  L_0 g - A^{-1}g \,=\, L g - \bigl(A^{-1}+V)g \,=\, \cR\,, \quad
  \hbox{where}\quad \cR \,=\, (\mu-1)A^{-1}g - Vg\,.
\end{equation}
Since $r\cR\phi = r\bigl(\phi(L_0 g) -g(L_0\phi)\bigr) = \frac{\D}{\D r}
\bigl(r(\phi'g- g'\phi)\bigr)$, we have for $r_1 > r_0 > 0$ the identity
\begin{equation}\label{STrel}
  \int_{r_0}^{r_1} \cR(r)\phi(r)r\dd r \,=\, r\Bigl(\phi'(r)g(r) - 
  g'(r)\phi(r)\Bigr)\,\Big|_{r=r_0}^{r=r_1}\,.
\end{equation}
Now, we let $r_0$ tend to $0$ and $r_1$ to $+\infty$ along appropriate
sequences, in such a way that the right-hand side of \eqref{STrel} converges
to zero. This possible, because we know that $\phi(r) = \cO(r)$ and
$\phi'(r) = \cO(1)$ as $r \to 0$, while $\phi(r) = \cO(1/r)$ and $\phi'(r) = 
\cO(1/r^2)$ as $r \to +\infty$; moreover the behavior of $g$ in these
limits is given in Lemma~\ref{gODElem} and Remark~\ref{rem:gadd}.  We thus
deduce from \eqref{STrel} that $\int_0^\infty \cR \phi r\dd r = 0$, which is
impossible if the function $\cR$ has a constant sign and is not identically
zero. So, if $V$ does not change sign, we must have $\mu \ge 1$ if $V \ge 0$ and
$\mu \le 1$ if $V \le 0$; moreover $\mu = 1$ is possible only if $V \equiv 0$.
Since $\mu = C_H^{-1}$, this gives the desired conclusion.
\end{proof}

\begin{rem}\label{Hardyscale}
As is easily verified, the optimal constant $C_H$ in Hardy's inequality
\eqref{Hardy} is unchanged if the function $A(r)$ is replaced by $\lambda^{-2}
A(\lambda r)$ for some $\lambda > 0$. This corresponds to a rescaling
of the vortex profile $\omega_*$.
\end{rem}

We now give two important examples where the sign of $C_H - 1$ can be determined. 

\medskip\noindent{\bf Example 1 : Algebraic vortex.} Given $\kappa > 1$, 
we define
\begin{equation}\label{algvort}
  \omega_*(r) \,=\, \frac{1}{(1+r^2)^\kappa}\,, \qquad 
  \psi_*'(r) \,=\, \frac{1}{2(\kappa{-}1)r}\Bigl(1 - \frac{1}{(1+r^2)^{\kappa-1}}
  \Bigr)\,.
\end{equation}
We have
\[
  A(r) \,=\, -\frac{\psi_*'(r)}{\omega_*'(r)} \,=\, \frac{1}{4\kappa(\kappa{-}1)r^2}
  \Bigl((1+r^2)^{\kappa+1}-(1+r^2)^2\Bigr)\,.
\]
When $\kappa = 2$ (Kaufmann-Scully vortex), inequality \eqref{Hardy} holds
with optimal constant $C_H = 1$, and is saturated for $f(r) = r^2/(1+r^2)^2$.
Indeed, it is easy to verify that $A(r) = (1+r^2)^2/8$ and $V(r) = 0$ in that 
particular case. Taking $g(r) = r/(1+r^2)$, a direct calculation  shows that 
$Lg = A^{-1}g$, so that $C_H = 1$. 

If $\kappa > 2$, we prove in Section~\ref{ssecA2} that the potential $V$ 
is positive, so that $C_H < 1$ by Corollay~\ref{Hardycor}. Finally, if 
$1 < \kappa < 2$, the potential $V$ is negative, implying that $C_H > 1$. 
Summarizing, for the family of algebraic vortices \eqref{algvort}, the 
quadratic form $J$ is coercive on $X_\rs \cap X_0$ if and only if $\kappa > 2$. 

\medskip\noindent{\bf Example 2 : Gaussian vortex.} We next consider
the Oseen vortex given by 
\begin{equation}\label{gaussvort}
  \omega_*(r) \,=\, e^{-r^2/4}\,, \qquad \psi_*'(r) \,=\,  \frac{2}{r}
  \Bigl(1 - e^{-r^2/4}\Bigr)\,, \qquad A(r) \,=\, \frac{4}{r^2}\Bigl(e^{r^2/4} 
  - 1\Bigr)\,.
\end{equation}
In that case too, the potential $V$ defined in \eqref{Vdef} is positive, see
Section~\ref{ssecA2}. By Corollary~\ref{Hardycor}, we conclude that $C_H < 1$,
so that the quadratic form $J$ is coercive on $X_\rs \cap X_0$. A numerical
calculation gives the approximate value $C_H \approx 0.57$, so that
$\gamma \approx 0.43$.

\subsection{The quadratic form $J$ without mass constraint}\label{ssec23}

In this short section we make a few remarks on the index of the quadratic form
\eqref{JdefX} when considered on the whole space $X$ defined by \eqref{Xdef},
and not only on the subspace $X_0$ given by \eqref{X0def}. Our first observation
is that, due to lack of scale invariance in this context, the form $J$ cannot be
positive on $X$ if the underlying steady state $\bar \omega$ is sharply
concentrated near the origin.  To see this, we consider the rescaled vortex
$\bar\omega_\lambda(x) = \lambda^2 \bar\omega(\lambda x)$ and the associated
weight function $\cA_\lambda(x) = \lambda^{-2}\cA(\lambda x)$, see
Remark~\ref{Hardyscale}.  We denote by $J_\lambda$ the quadratic form on $X$
corresponding to the steady state $\bar\omega_\lambda$, namely the form
\eqref{JdefX} where $\cA$ is replaced by $\cA_\lambda$. If $\omega \in X$ and
$\omega_\lambda(x) = \lambda^2 \omega(\lambda x)$, a simple calculation shows
that
\[
  J_\lambda(\omega_\lambda) \,=\, J(\omega) - \frac{M_0^2}{4\pi}\,\log(\lambda)\,,
  \qquad \hbox{where}\quad M_0 \,=\, \int_{\R^2} \omega(x)\dd x\,.
\]
If $M_0 \neq 0$, it is clear that $J_\lambda(\omega_\lambda) < 0$ when $\lambda > 0$ 
is sufficiently large, so that the quadratic form $J_\lambda$ cannot be positive
in this regime. 

We next argue that, for any vortex $\bar\omega$ satisfying Hypotheses~\ref{omegahyp}, 
the index of the quadratic form is well defined in the sense that $J$ has (at
most) a finite number of negative directions. In view of Theorem~\ref{Jprop1}, 
it is sufficient to evaluate $J$ on radially symmetric functions $\omega \in X_\rs$. 
The following expression will be useful: 

\begin{lem}\label{lem:radJ}
For any $\omega \in X_\rs$, we have
\begin{equation}\label{Jrad}
  J(\omega) \,=\, \pi \int_0^\infty A(r)\om(r)^2r\dd r \,+\, \pi\int_0^\infty\!\!
  \int_0^\infty \log\bigl(\max(r,s)\bigr) \,r\om(r)s\om(s)\dd r\dd s\,.
\end{equation}
\end{lem}

\begin{proof}
Here and below, with a slight abuse of notation, we consider any $\omega \in X_\rs$ 
as a function of the one-dimensional variable $r = |x|$. For such vorticities, 
the first integral in \eqref{JdefX} obviously gives the first term in \eqref{Jrad}, 
so it remains to establish the following expression of the energy: 
\begin{equation}\label{Erad}
  E(\omega) = - \pi\int_0^\infty\!\!\int_0^\infty \log\bigl(\max(r,s)\bigr)\, 
  r\om(r)\,s\om(s)\dd r\dd s\,, \qquad \omega \in X_\rs\,.
\end{equation}
To this end, we introduce polar coordinates $x = r\,e^{i\theta}$, $y = s\,e^{i\zeta}$
to compute the right-hand side of \eqref{Edef}, and we use the identity
\begin{equation}\label{anglesid}
  \int_0^{2\pi}\!\!\int_0^{2\pi} \log|re^{i\theta}-se^{i\zeta}|\dd\theta\dd \zeta \,=\, 
  2\pi \int_0^{2\pi}\log|re^{i\theta}-s|\dd\theta  \,=\, 4\pi^2 \log\bigl(
  \max(r,s)\bigr)\,.
\end{equation}
The formula \eqref{anglesid} is well known and can be derived in many ways. For
example, assuming that $r$ is a fixed positive number, we interpret the last
integral as a function of $s\in\C$. This expression obviously depends only on
$|s|$, is continuous everywhere, and is analytic both inside and outside of the circle
$|s|=r$. Inside the circle it has to be constant and outside the circle it
coincides with the potential of a point particle of mass $2\pi$ located at the
origin, which is $2\pi\log|s|$. This gives \eqref{anglesid}, and \eqref{Erad}
follows.
\end{proof}

Applying the change of variables $w(r) = \om(r)A(r)^{1/2}$, so that $w \in 
Y = L^2(\R_+,r\dd r)$ when $\omega \in X_\rs$, the formula \eqref{Jrad} becomes
\begin{equation}\label{Jrad2}
  \frac{1}{\pi}\,J(\omega) \,=\, \int_0^\infty w(r)^2r\dd r \,-\, \int_0^\infty
  \!\!\int_0^\infty \ck(r,s)w(r)w(s) rs\dd r\dd s\,,
\end{equation}
where $\ck(r,s) = -\log\bigl(\max(r,s)\bigr)\,A(r)^{-1/2}\,A(s)^{-1/2}$.  Under
Hypotheses~\ref{omegahyp}, we have the lower bound $A(r) \ge C(1+r)^\beta$ for
some $\beta > 2$, which implies that
\[
  \int_0^\infty \!\!\int_0^\infty \ck(r,s)^2 \,rs \dd r\dd s\,<\, \infty\,.
\]
This means that the right-hand side of \eqref{Jrad2} is the quadratic form in
$Y$ associated with a selfadjoint operator of the form $\1 - \cK$, where $\1$
is the identity and $\cK$ is a Hilbert-Schmidt perturbation. By compactness,
this operator has (at most) a finite number of negative eigenvalues, which
means that the index of the quadratic form $J$ on $X$ is well defined.

The eigenvalues of $\cK$ can also be thought of as eigenvalues of the quadratic
form \eqref{Erad} with respect to the reference form $\omega \mapsto \pi
\int_0^\infty A(r)\om(r)^2r\dd r$. As is easily verified, if $\lambda$ is 
such an eigenvalue, the corresponding eigenfunction $\omega$ satisfies
\begin{equation}\label{psieig}
  -\psi(r) \,=\, \lambda A(r)\om(r)\,,\qquad \hbox{where}\quad 
  \psi(r) \,=\, \int_0^\infty \log\bigl(\max(r,s)\bigr)\,s\omega(s) 
  \dd s\,.  
\end{equation}
Since $\omega(r) = \psi''(r) + \frac{1}{r}\,\psi'(r)$, the first relation
in \eqref{psieig} is an ordinary differential equation for the stream
function $\psi : \R_+\to\R$, to be solved with the boundary conditions
\[
  \psi'(0) \,=\, 0\,, \qquad \hbox{and}\qquad \lim_{r \to +\infty} \Bigl(\psi(r)
  \log(2r) - \psi(2r)\log(r)\Bigr) \,=\, 0\,,
\]
which can be deduced from the expression of $\psi$ in \eqref{psieig}. 
For the Lamb-Oseen vortex~\eqref{gaussvort} a numerical computation gives the
largest eigenvalue $\lambda\approx 0.7127$, thus suggesting that the form $J$ is
strictly positive definite on the whole space $X_\rs$ in that case.
In contrast, the largest eigenvalue for the algebraic vortices~\eqref{algvort} 
seems to exceed the threshold value $1$, indicating that for those 
vortices the form $J$ is not positive definite without additional constraints 
on $\om$.

\subsection{The maximal energy as a function of the constraints}\label{ssec24}

In Section~\ref{ssec11} we considered the classical problem of maximizing a
function $f : \R^n \to \R$ under a family of constraints of the form
$g_1 = c_1, \dots, g_m = c_m$, where $g_1, \dots, g_m : \R^n \to \R$. Given
$c = (c_1,\dots,c_m) \in \R^m$, we recall the notation
$\Xc = \{x\in\R^n\,;\,g_1(x)=c_1,\dots,g_m(x)=c_m\}$.  Assuming that $f$ reaches
a non-degenerate maximum on $\Xc$ at some point $\bar x \in \Xc$ where the
first-order derivatives $g_1'(\bar x), \dots, g_m'(\bar x)$ are linearly
independent, we introduced the quadratic form $\cQ$ defined by \eqref{I5}, which
is the second order differential of the Lagrange function \eqref{I9} at
$\bar x$. In the present section, we are interested in the index of the form
$\cQ$ on larger subspaces than $\TXc$. As was already mentioned, this question
is closely related to concavity properties of the function $M$ defined by
\eqref{Mf} or, almost equivalently, to convexity properties of the set
$S=\{(g_1(x),\dots,g_m(x),f(x))\,;\,x\in\R^n\}\subset\R^{m+1}$ near its ``upper
boundary''.

The situation becomes particularly transparent if we use adapted coordinates
which, as it turns out, have a fairly complete analogy in 2d Euler case. 
Let us assume that we can introduce new coordinates $(c_1,\dots,c_m,y_1,\dots,y_{n-m})$ 
in $\R^n$ such that, as before, $c_1,\dots,c_m$ are the values of the 
constraints $g_1,\dots,g_m$, and the additional coordinates $y_1,\dots,y_{n-m}$ 
are chosen so that the points having coordinates $(c_1,\dots,c_m,0,\dots,0)$ 
are those where $f$ attains its maximum on $\Xc$.\footnote{In a non-degenerate 
situation, the local existence of such a coordinate system is clear by standard 
arguments, but globally the situation can, of course, be more complicated.}
Denoting $M(c_1,\dots,c_m) = f(c_1,\dots,c_m,0,\dots,0)$ as in \eqref{Mf},
one verifies that
\begin{equation}\label{Mcj} 
  \frac{\partial M}{\partial c_j}(c_1,\dots,c_m) \,=\, \lambda_j\,, 
  \qquad j = 1,\dots,m\,,
\end{equation}
where $\lambda_1, \dots, \lambda_m$ are the Lagrange multipliers introduced 
in \eqref{I4}. Moreover the extremality condition on $\Xc$ implies that
\[
  \frac{\partial f}{\partial y_k}(c_1,\dots,c_m,0,\dots,0) \,=\, 0\,, \qquad
  k=1,\dots,n-m\,.
\]
We infer that
\begin{equation}\label{D2f} 
  D^2f(c_1,\dots,c_m,0,\dots,0) ~=~ \begin{pmatrix} \left(\frac{\partial^2f}{
  \partial c_i\partial c_j}\right)_{i,j=1}^m &  0\\ 0 & \left(\frac{\partial^2 f}{
  \partial y_k\partial y_\ell}\right)_{k,\ell=1}^{n-m}\end{pmatrix}\,,
\end{equation}
where all derivatives are evaluated at the point $(c_1,\dots,c_m,0,\dots,0)$. 
The first submatrix in the right-hand side of \eqref{D2f} is precisely the 
Hessian of $M$, and the second submatrix is always negative definite, due to our 
assumption that $f$ reaches a maximum at $(y_1,\dots,y_{n-m}) = (0,\dots,0)$ 
for any fixed value of $c_1,\dots,c_m$. So we conclude that the quadratic form 
$\cQ$ defined in \eqref{I5} is negative definite at $\bar x$ if and only if 
the Hessian of $M$ is negative definite at $(c_1,\dots,c_m)$, where $c_j = 
g_j(\bar x)$ for $j = 1, \dots,m$. 

Another interesting object is the function
\begin{equation}\label{l2}
\begin{split}
  N(\lambda_1,\dots,\lambda_m) \,&=\, \sup_{x \in \R^n}\Bigl(f(x)-\lambda_1g_1(x) -
  \ldots -\lambda_m g_m(x)\Bigr) \\
  \,&=\, \sup_{c \in \R^m}\Bigl(M(c_1,\dots,c_m) - \lambda_1c_1 - \ldots
  -\lambda_m c_m\Bigr)\,,
\end{split}
\end{equation}
which is the {\em Legendre transform} of $M$. Under appropriate assumptions, 
the main one being the concavity of $M$, this quantity is well defined and
the relation \eqref{Mcj} can be inverted (at least locally) via the formula  
\begin{equation}\label{l30}
  c_j \,=\, -\frac{\partial N}{\partial\lambda_j}(\lambda_1,\dots,\lambda_m)\,,
  \qquad j = 1,\dots,m\,. 
\end{equation}

\medskip We now return to the infinite-dimensional framework of the 2d Euler
equation, with the manifold $\R^n$ replaced by the phase space $\cP$ introduced
in Section~\ref{ssec12}, the function $f$ replaced by the energy $E$ in
\eqref{Edef}, the constraints $g_j$ replaced by the Casimir functionals
$h(a,\om)$ in~\eqref{ni1}, and the submanifolds $\Xc$ replaced by the orbits
$\cO_{\om}$ in~\eqref{ni2}. In that case we have
\begin{equation}\label{max3}
  \max_{\om \in\cO_{\bar \om}}E(\om) \,=\, E(\bar\om^*)\,,
\end{equation}
where, as before, $\bar\om^*$ denotes the symmetric decreasing rearrangement of
an element $\bar\om \in \cP$. As $\cO_{\bar\om}$ is characterized in terms of 
the functionals $h(a,\om)$ defined in \eqref{ni1}, the energy of the maximizer 
$\bar\om^*$ in $\cO_{\bar\om}$ can also be expressed in terms of the constraint 
function $a\to h(a,\bar\om)$. It turns out that the representation formula is
quite explicit. 

\begin{prop}\label{en-h}
Given $\bar\om\in \cP$, we define $\ch(a) = \pi^{-1} h(a,\bar\om) = \pi^{-1}
|\{\bar\om>a\}|$ for any $a > 0$. Then
\begin{equation}\label{energy-via-h}
  \cE(\ch) ~:= \!\max_{\genfrac{}{}{0pt}{1}{\om\in\cP}{h(\cdot,\om)=\pi\ch}} \!E(\om) 
  \,=~ \frac{\pi}{8} \int_0^m\!\!\int_0^m L(\ch(a),\ch(b))\dd a\dd b \,+\, 
  \frac1{8\pi}M_0^2\,, 
\end{equation}
where $m=\max\bar \om$, $M_0=\irt\bar\om\dd x=\pi\int_0^m \ch(a)\dd a$, and
\begin{equation}\label{nLdef}
  L(R,S) \,=\, -RS\log\max(R,S)-\frac 12\min(R,S)^2\,.
\end{equation}
\end{prop}

\begin{proof}
Replacing $\bar\omega$ with $\bar\omega^*$ (an operation that does not affect 
the function $\ch$), we can assume that $\bar\omega$ is radially symmetric 
and nonincreasing in the radial direction. In view of of \eqref{max3}, we then 
have $\cE(\ch) = E(\bar\omega)$, and if we consider $\bar\om$ as a function 
of the radius $r = |x|$ we observe that $\ch(a) = (\bar\om^{-1}(a))^2$ for 
any $a > 0$. To compute $E(\bar\om)$, we start from the expression \eqref{Erad}, 
and we introduce the functions
\[
  k(r,s) \,=\, -rs\log\max(r,s)\,, \qquad K(R,S) \,=\, L(R,S) + RS\,.
\]
Clearly $K(R,0)=0\,,\,K(0,S)=0$ for $R,S>0$, and one can verify by direct 
calculation that $K(R,S)$ is twice continuously differentiable on $(0,\infty)
\times(0,\infty)$ with
\[
  \frac{\partial^2K}{\partial R\partial S}(R,S) \,=\, -\log\max(R,S)\,,
  \qquad R,S > 0\,.
\]
Similarly the function $(r,s) \mapsto K(r^2,s^2)$ is twice continuously 
differentiable on $[0,\infty)\times[0,\infty)$ and
\[
  \frac{1}{8}\,\frac{\partial^2}{\partial r\partial s}\,K(r^2,s^2) \,=\, k(r,s)\,.
\]
Integrating by parts in \eqref{Erad} and recalling that $m = \max\bar\om$, 
we can thus write
\begin{equation}\label{max50}
\begin{split}
  E(\bar\om) \,&=\, \frac{\pi}{8}\int_0^\infty\!\!\int_0^\infty\frac{\partial^2 }{
  \partial r \partial s}K(r^2,s^2)\,\bar\om(r)\,\bar\om(s)\dd r\dd s \,=\,\frac{\pi}{8}
  \int_0^\infty\!\!\int_0^\infty K(r^2,s^2) \dd\bar\om(r) \dd\bar\om(s)\\
  \,&=\, \frac{\pi}8 \int_0^m\!\!\int_0^m K\bigl((\bar\om^{-1}(a))^2,(\bar\om^{-1}(b))^2\bigr)
  \dd a\dd b \,=\, \frac{\pi}8\int_0^m\!\!\int_0^m K(\ch(a),\ch(b))\dd a \dd b\,\\
  \,&=\,\frac{\pi}{8}\int_0^m\!\!\int_0^m L(\ch(a),\ch(b))\dd a\dd b
  + \frac{1}{8\pi}M_0^2\,,
\end{split}
\end{equation}
where we have formally used the substitutions $\bar\om(r)=a$, $\bar\om(s)=b$. 
This is straighforward when $\bar\om$ is strictly decreasing, and the general
case where $\bar\om$ is nonincreasing can be treated by integrating only over
the intervals where $\bar\om$ is strictly decreasing.
\end{proof}

We now make a more precise comparison with the finite-dimensional situation
above. Let us assume that $\bar \om\in\cP$ is radially symmetric with
$\partial_r\bar\om(r)<0$ for all $r > 0$ and $\partial_r^2\bar \om(0) < 0$.  To
eliminate the translational symmetries, we work with the manifold
\begin{equation}\label{tcPdef}
  \tilde\cP \,=\, \bigl\{\om\in\cP\,;\,M_0(\om) = M_0(\bar \om)\,,\,
  M_j(\om) = 0\,,\,j=1,2\bigr\}\,,
\end{equation}
where $M_0, M_j$ are as in \eqref{M0}, \eqref{MjIdef}. If $\eta\in \cX_1$
(see~\eqref{X2def}) is smooth and compactly supported with suffiently small
$C^2$ norm, then $\bar\om + \eta\in\tilde\cP$. Denoting by $\eta_\rs$ the 
projection of $\eta$ onto the subspace $X_\rs$ defined in \eqref{Xrsdef}, 
we can take the quantities $h(a,\bar\om+\eta_\rs)$ and $\eta^\perp_\rs := \eta 
- \eta_\rs$ as the (approximate) analogues of the coordinates $c_j$ and $y_k$, 
respectively. The analogy is not perfect, due to the stronger-than-ideal
assumptions on $\eta$, but it is sufficient for concluding that when
$\bar\om=\bar\om^*$, the negative-definiteness of Arnold's form~\eqref{ni9} 
on the tangent space $T_{\bar\om}\tilde\cP$  is strongly related to
the concavity of the energy $E$ in the variable\footnote{It is perhaps
 worth recalling that $E$ is convex in $\om$ on the subspace given by
 $\irt\om \dd x=0$. However, at some points it may be concave in $\ch$, at least on
 the subspace given by $\int_0^\infty \ch(a)\dd a=0 $.} $\ch$ at the function
$\bar\ch(a)=\pi^{-1} h(a,\bar\om)$. In some sense the expression~\eqref{energy-via-h} 
is ``trying to be concave'', although not quite achieving this: the function 
$L(R,S)$ is separately concave, but not concave. The second variation on the 
space $X_0$ is given by the quadratic form which takes a function $\xi(a)$ with 
$\int_0^m\xi(a)\dd a=0$ to
\[
  \frac{\pi} 8\int_0^m\!\!\int_0^m\left(D_1^2L(\ch(a),\ch(b))\xi(a)^2 + 2D_1D_2
  L(\ch(a),\ch(b))\xi(a)\xi(b)+D_2^2L(\ch(a),\ch(b))\xi(b)^2\right)\dd a\dd b\,.
\]
Due to the separate concavity of $L$ the first term and the third term are
negative, but the second one can lead to the form being indefinite. In view of
our previous considerations, the negativity of the form is equivalent to
the validity of the Hardy inequality~\eqref{Hardy} with $C_H < 1$, and it 
is not hard to verify directly that this is indeed the case. As an analogue 
of~\eqref{Mcj}, we also note that the variational derivative of $\cE$ with 
respect to $\ch$ is
\begin{equation}\label{l50}
  \frac{1}{\pi}\,\frac{\delta\cE}{\delta\ch}(a) \,=\, \Lambda(a) \,=\, 
  -\Phi'(a)\,.
\end{equation}
We will not go into the details as we will not work with this expression. The
reader can also derive the analogue of~\eqref{l30} (under appropriate
assumptions).

\section{Global maximization of the free energy}
\label{sec3}

In the previous section we observed that some radially symmetric vortices
$\bar\omega$, including the Gaussian vortex \eqref{gaussvort} and the algebraic
vortex \eqref{algvort} with $\kappa > 2$, are non-degenerate local maxima of the
associated free energy functional \eqref{cEdef} once restricted to the manifold
$\tilde\cP$ defined in \eqref{tcPdef}.  This was established by showing that the
second order differential $F''(\bar\omega)$ is strictly negative definite on the
tangent space $T_{\bar\omega}\tilde\cP$. We now follow a different approach,
which relies on the direct method in the calculus of variations. Under
appropriate assumptions on the function $\Phi$ in \eqref{cEdef}, we show that
the free energy $F(\omega)$ has a global maximum on the set of all vorticity
distributions with a fixed mass $M$. By construction, if $\bar\omega$ is any
maximizer obtained in this way, the conclusion of Theorem~\ref{Jprop2} applies
with $\gamma \ge 0$, so that Hardy's inequality \eqref{Hardy} holds with
$C_H \le 1$. Note also that, according to the discussion in
Section~\ref{ssec24}, prescribing $\Phi$ amounts to fixing the ``Lagrange
multipliers'' in our constrained maximization problem.

We start with a preliminary result, which is probably well known. For the 
reader's convenience, the proof is reproduced in Section~\ref{ssecA1}. 

\begin{prop}\label{prop:logf}
Assume that $f \in L^1(\R^n)$ is nonnegative and that $M := \int_{\R^n} 
f(x)\dd x > 0$. Then
\begin{align}\label{logf1}
  M \,+\, \int_{\R^n} \bigl(\log_-|x|\bigr) \,f(x)\dd x ~&\lesssim~
  M \,+\, \int_{\R^n} \Bigl(\log_+\frac{f(x)}{M}\Bigr) \,f(x)\dd x\,, \\
  \label{logf2}
  M \,+\, \int_{\R^n} \bigl(\log_+|x|\bigr) \,f(x)\dd x ~&\gtrsim~
  M \,+\, \int_{\R^n} \Bigl(\log_-\frac{f(x)}{M}\Bigr) \,f(x)\dd x\,,
\end{align}
where the implicit constants only depend on the space dimension $n$. 
Moreover, if $f$ is radially symmetric and nonincreasing in the 
radial direction, then the reverse inequalities also hold. 
\end{prop}

We next specify the function space in which we shall solve our maximization 
problem. 

\begin{df}\label{XMdef}
Given any $M > 0$, we denote by $X_M$ the set of all $\omega \in L^1(\R^2)$ such 
that $\omega(x) \ge 0$ for almost all $x \in \R^2$ and 
\begin{equation}\label{XMprop}
  \int_{\R^2} \omega(x)\dd x \,=\, M\,, \quad
  \DS \int_{\R^2} \omega(x) \log\bigl(1+|x|\bigr)\dd x \,<\, \infty\,, \quad
  \DS \int_{\R^2} \omega(x) \log\bigl(1+\omega(x)\bigr)\dd x \,<\, \infty\,.
\end{equation}
\end{df}

For later use we observe that, if $\omega \in X_M$ and if $\omega^*$ denotes
the symmetric nonincreasing rearrangement of $\omega$, then $\int_{\R^2} 
\omega^*(x) \dd x = \int_{\R^2} \omega(x) \dd x = M$ and
\begin{align*}
  &\DS \int_{\R^2} \omega^*(x) \log\bigl(1+|x|\bigr)\dd x \,\le\, 
  \DS \int_{\R^2} \omega(x) \log\bigl(1+|x|\bigr)\dd x \,<\, \infty\,, \\[1mm]
  &\DS \int_{\R^2} \omega^*(x) \log\bigl(1+\omega^*(x)\bigr)\dd x \,=\,
  \DS \int_{\R^2} \omega(x) \log\bigl(1+\omega(x)\bigr)\dd x \,<\, \infty\,.
\end{align*}
This shows that the set $X_M \subset L^1(\R^2)$ is invariant under the action
of the symmetric nonincreasing rearrangement. 

For $\omega \in X_M$, we consider the free energy defined by $F(\omega) = 
E(\omega) + S(\omega)$, where
\[
  E(\omega) \,=\, \frac{1}{4\pi} \int_{\R^2} \int_{\R^2} \log\frac{1}{|x-y|}\,
  \omega(x)\omega(y)\dd x\dd y\,, \qquad
  S(\omega) \,=\, \int_{\R^2} \Phi(\omega(x))\dd x\,.
\]
We have shown in Proposition~\ref{prop:Edef} that the energy $E(\omega)$ 
is finite for any $\omega \in X_M$. Unlike in Section~\ref{sec2}, the 
function $\Phi$ in the entropy term is not related here to any radially 
symmetric vortex, but is an arbitrary function satisfying the following 
properties: 

\begin{hyp}\label{hypPhi}
The function $\Phi : [0,+\infty) \to \R$ is continuous with $\Phi(0) = 0$. Moreover, 
there exists constants $C_1 \in \R$, $C_2 < M/(8\pi)$, and $C_3 > M/(8\pi)$ such that
\begin{equation}\label{Phicond1}
\begin{split}
  \Phi(\omega) \,&\le\, C_1 \omega + C_2\,\omega \log\frac{M}{\omega} 
  \quad \hbox{when }\, \omega \le M\,, \\[1mm]
  \Phi(\omega) \,&\le\, C_1 \omega - C_3\,\omega \log\frac{\omega}{M} 
  \quad \hbox{when }\, \omega \ge M\,.
\end{split}
\end{equation} 
\end{hyp}

Under Hypotheses~\ref{hypPhi}, the positive part of $\Phi$ satisfies 
$\Phi_+(\omega) \le C\omega(1+|\log(\omega/M)|)$ for some constant $C > 0$, 
and this implies in particular that the entropy $S(\omega)$ is well defined 
in $\R \cup \{-\infty\}$ for any $\omega \in X_M$. We are now in a position 
to state the main result of this section.

\begin{thm}\label{thm:globmax}
Fix any $M > 0$. Under Hypotheses~\ref{hypPhi}, there exists $\bar\omega \in X_M$
such that
\[
  F(\bar \omega) \,=\, E(\bar \omega) + S(\bar \omega) \,=\, \sup_{\omega \in X_M} 
  \bigl(E(\omega) + S(\omega)\bigr)\,.
\]
Moreover $\bar\omega$ can be chosen to be radially symmetric and nonincreasing 
in the radial direction. 
\end{thm}

The proof of Theorem~\ref{thm:globmax} is divided into two parts. The first
one consists in showing that the free energy $F$ is bounded from above
on $X_M$, and that there exists a maximizing sequence which is convergent
in $L^1(\R^2)$. We formulate this in a separate statement:

\begin{prop}\label{prop:globmax}
Under Hypotheses~\ref{hypPhi}, the free energy $F = E+S$ is bounded from above 
on the space $X_M$:
\[
  F_M \,:=\, \sup_{\omega \in X_M} \bigl(E(\omega) + S(\omega)\bigr) \,<\, \infty\,.
\]
Moreover, there exists a maximizing sequence $(\omega_j)_{j \in \N}$ in $X_M$ which
converges in $L^1(\R^2)$ to some limiting profile $\bar\omega = \bar\omega^* \in X_M$ 
as $j \to +\infty$, and we have $S(\bar\omega) > -\infty$. 
\end{prop}

\begin{proof}
Our starting point is the logarithmic Hardy-Littlewood-Sobolev inequality 
\begin{equation}\label{logHLS}
  E(\omega) + \frac{M}{8\pi}\int_{\R^2} \omega \log\frac{M}{\omega}\dd x \,\le\, 
  \frac{M^2}{8\pi}\,\bigl(1+\log\pi\bigr)\,,
\end{equation}
which holds for all $\omega \in X_M$, see \cite{Carlen-Loss}. In view of 
\eqref{Phicond1}, we deduce from \eqref{logHLS} that
\begin{equation}\label{Eupper}
\begin{split}
  E(\omega) &+ S(\omega) + \Bigl(\frac{M}{8\pi} - C_2\Bigr)\int_{\omega < M}
  \omega \log\frac{M}{\omega}\dd x + \Bigl(C_3 - \frac{M}{8\pi}\Bigr)\int_{\omega > M}
  \omega \log\frac{\omega}{M}\dd x \\  
  \,&\le\, E(\omega) + C_1 M + \frac{M}{8\pi}\int_{\R^2} \omega \log\frac{M}{\omega}\dd x 
  \,\le\, C_1 M + \frac{M^2}{8\pi}\,\bigl(1+\log\pi\bigr)\,.
\end{split}
\end{equation}
Since $C_2 < M/(8\pi)$ and $C_3 > M/(8\pi)$, this proves that $F_M \le C_1 M + 
M^2(1+\log\pi)/(8\pi)$. 

Now, let $(\omega_j)_{j \in \N}$ be a sequence in $X_M$ such that $E(\omega_j) + 
S(\omega_j) \to F_M$ as $j \to +\infty$. If we denote by $(\omega_j)^* \in X_M$ 
the symmetric nonincreasing rearrangement of $\omega_j$, we know that $E((\omega_j)^*) 
\ge E(\omega_j)$ and $S((\omega_j)^*) = S(\omega_j)$ for all $j \in \N$, so that 
$\bigl((\omega_j)^*\bigr)_{j \in \N}$ is a fortiori a maximizing sequence. So we 
assume henceforth that $\omega_j = (\omega_j)^*$, i.e. $\omega_j$ is radially 
symmetric and nonincreasing in the radial direction. In that case, there 
exists a constant $C_0 > 0$ such that
\begin{equation}\label{omjapriori}
  \int_{\R^2} \omega_j(x)\,\Bigl|\log\frac{\omega_j(x)}{M}\Bigr| \dd x \,\le\, C_0\,,
  \qquad \hbox{and}\qquad 
  \int_{\R^2} \omega_j(x)\,\bigl|\log|x|\bigr| \dd x \,\le\, C_0\,,
\end{equation}
for all $j \in \N$. Indeed, the first inequality in \eqref{omjapriori} follows
directly from \eqref{Eupper}, and the second one is consequence of the first 
inequality and of Proposition~\ref{prop:logf}, since $\omega_j = (\omega_j)^*$. 

It remains to verify that one can extract a subsequence, still denoted by 
$(\omega_j)_{j \in \N}$, such that $\|\omega_j - \bar\omega\|_{L^1} \to 0$ as 
$j \to +\infty$ for some $\bar\omega \in X_M$. This can be done using the 
concentration-compactness principle \cite{Lions}. Indeed, among the three possible 
scenarios:

\noindent \quad i) Concentration is ruled out by the first inequality in 
\eqref{omjapriori}; 

\noindent \quad ii) Evanescence is ruled out by the second inequality in 
\eqref{omjapriori};

\noindent \quad iii) Dichotomy cannot occur for radially symmetric nonincreasing
functions. 

\smallskip\noindent So there exists a subsequence $(\omega_j)_{j \in \N}$ that
converges in $L^1(\R^2)$, and pointwise almost everywhere, to some limit
$\bar\omega \in L^1(\R^2)$. It is clear that $\bar \omega \ge 0$ and
$\int_{\R^2} \bar \omega(x)\dd x = M$. Moreover, using \eqref{omjapriori} and
Fatou's lemma, we obtain
\begin{equation}\label{barombd}
  \int_{\R^2} \bar\omega(x)\,\Bigl|\log\frac{\bar\omega(x)}{M}\Bigr| \dd x \,\le\, 
  C_0\,, \qquad \hbox{and}\qquad 
  \int_{\R^2} \bar\omega(x)\,\bigl|\log|x|\bigr| \dd x \,\le\, C_0\,,
\end{equation}
which proves that $\bar\omega \in X_M$. Finally, if we decompose $\Phi = 
\Phi_+ - \Phi_-$, where $\Phi_+, \Phi_-$ denote the positive and negative parts
of $\Phi$, we have the lower bound
\begin{equation}\label{Slower}
  S(\bar\omega) \,\ge\, -\int_{\R^2}\Phi_-(\bar\omega(x))\dd x \,\ge\, 
  - \liminf_{j \to +\infty}\int_{\R^2}\Phi_-(\omega_j(x))\dd x\,,
\end{equation}
where the second inequality is again obtained by Fatou's lemma. But we have 
the identity
\[
  \int_{\R^2}\Phi_-(\omega_j(x))\dd x \,=\, \int_{\R^2}\Phi_+(\omega_j(x))\dd x
  - S(\omega_j) \,=\, \int_{\R^2}\Phi_+(\omega_j(x))\dd x + E(\omega_j) - 
  F(\omega_j)\,,
\]
where the first two terms in the right-hand side are bounded uniformly in 
$j$ by \eqref{omjapriori}, whereas $F(\omega_j)$ is bounded from below 
since $(\omega_j)$ is a maximizing sequence for $F$. We conclude that
the right-hand side of \eqref{Slower} is finite, so that $S(\bar\omega) 
> -\infty$.
\end{proof}

To conclude the proof of Theorem~\ref{thm:globmax}, it remains to show that 
the free energy is upper semicontinuous along the maximizing sequence 
constructed in Proposition~\ref{prop:globmax}, namely
\begin{equation}\label{uppersc}
  E(\bar \omega) + S(\bar \omega) \,\ge\, \limsup_{j \to +\infty} 
  \Bigl(E(\omega_j) + S(\omega_j)\Bigr) \,=\, F_M\,.
\end{equation}
This immediately implies that $E(\bar \omega) + S(\bar \omega) = F_M$, 
which is the desired result. 

\begin{proof}[\bf Proof of Theorem~\ref{thm:globmax}] Let $(\omega_j)_{j \in \N}$
be the maximizing sequence defined in Proposition~\ref{prop:globmax}, 
and $\bar\omega \in X_M$ be the limiting profile. Given any sufficiently 
large $R > 0$, we decompose
\begin{align*}
  \omega_j(x) \,&=\, \omega_j(x) \,\1_{\{|x| \le R\}} + \omega_j(x) \,\1_{\{|x| > R\}} 
  \,=:\, \omega_{jR}^1(x) + \omega_{jR}^2(x)\,, \\
  \bar\omega(x) \,&=\, \,\bar\omega(x) \,\1_{\{|x| \le R\}} \,+\, \,\bar\omega(x) 
  \,\1_{\{|x| > R\}} \,=:\, \,\bar\omega_R^1(x) \,+ \bar\omega_R^2(x)\,,
\end{align*}
for all $x \in \R^2$. We thus have
\begin{align*}
  E(\omega_j) + S(\omega_j) \,&=\, E(\omega_{jR}^1) + S(\omega_{jR}^1) + 
  2 E(\omega_{jR}^1,\omega_{jR}^2) + E(\omega_{jR}^2) + S(\omega_{jR}^2)\,, \\ 
  E(\bar\omega) +\, S(\bar\omega) \,&=\, E(\bar\omega_R^1) \,\,+ S(\bar\omega_R^1) 
  \,\,+\,\, 2 E(\bar\omega_R^1,\bar\omega_R^2) \,\,+\, E(\bar\omega_R^2) 
  \,\,+\, S(\bar\omega_R^2)\,,
\end{align*}
where $E(\omega_1,\omega_2)$ is the bilinear form associated with the 
energy functional: 
\[
  E(\omega_1,\omega_2) \,=\, -\frac{1}{4\pi} 
  \int_{\R^2}\int_{\R^2} \log|x-y|\,\om_1(x)\,\om_2(y)\dd x\dd y\,.
\]
The upper-semicontinuity property \eqref{uppersc} can be deduced from
the following assertions: 
\begin{align}
  &\limsup_{j \to +\infty}\Bigl(E(\omega_{jR}^1) + S(\omega_{jR}^1)\Bigr) \,\le\, 
  E(\bar\omega_R^1) + S(\bar\omega_R^1)\,, \label{propa}\\
  &\sup_{j \in \N}\Bigl(2 E(\omega_{jR}^1,\omega_{jR}^2) + E(\omega_{jR}^2) + 
  S(\omega_{jR}^2)\Bigr) \,\le\, \delta_1(R) \,\xrightarrow[R \to +\infty]{}\, 0\,,
  \label{propb}\\[1mm]
  &2 E(\bar\omega_R^1,\bar\omega_R^2) +\, E(\bar\omega_R^2) + S(\bar\omega_R^2)
  \,=\, \delta_2(R) \,\xrightarrow[R \to +\infty]{}\, 0\,. \label{propc}
\end{align}
Indeed, assuming that \eqref{propa}--\eqref{propc} hold, we obtain
\[
  \limsup_{j \to +\infty}\Bigl(E(\omega_j) + S(\omega_j)\Bigr) - 
  \Bigl(E(\bar\omega) + S(\bar\omega)\Bigr) \,\le\, \delta_1(R) - \delta_2(R) 
  \,\xrightarrow[R \to +\infty]{}\, 0\,. 
\]

It remains to verify the assertions \eqref{propa}--\eqref{propc} above. We
recall that the functions $\omega_j, \bar \omega$ are radially symmetric and
nonincreasing in the radial direction. With a slight abuse of notation, we write
$\omega_j(r)$ instead of $\omega_j(x)$ when $r = |x|$, and similarly for
$\bar\omega$. Accordingly, using \eqref{Erad}, we obtain the following
expressions for the energy of $\omega_j$ and $\bar\omega$:
\begin{equation}\label{Erad2}
  E(\omega_j) = -\int_0^\infty M_j(r) \log(r) \,r \omega_j(r)\dd r\,, \qquad
  E(\bar\omega) = -\int_0^\infty \overline{M}(r)\log(r) \,r \bar\omega(r)\dd r\,,
\end{equation}
where 
\begin{equation}\label{Mjdef}
  M_j(r) \,=\, 2\pi\int_0^r s\omega_j(s)\dd s\,, \qquad 
  \overline{M}(r) \,=\, 2\pi\int_0^r s\bar\omega(s)\dd s\,, \qquad r > 0\,.
\end{equation}
Since $\omega_j \to \bar\omega$ in $L^1(\R^2)$, we see that $M_j(r) \to
\overline{M}(r)$ uniformly in $r$ as $j \to +\infty$.  Moreover, since $\omega_j
\in X_M$ satisfies \eqref{omjapriori}, the quantity $M_j(r)$ converges to $M$ as
$r \to +\infty$ uniformly in $j$. In particular, we can choose $R \ge 1$ large
enough so that $M_j(r) \ge M/2$ for all $j \in \N$ when $r \ge R$.

To prove \eqref{propa}, we first decompose
\[
  E(\omega_{jR}^1) - E(\bar\omega_R^1) \,=\, -\int_0^R \bigl(M_j(r) - 
  \overline{M}(r)\bigr)\log(r)\,r \omega_j(r)\dd r
  -\int_0^R \overline{M}(r) \log(r)\,r \bigl(\omega_j(r) - \bar\omega(r)\bigr)
  \dd r\,,
\]
and we deduce that
\begin{align}\nonumber
  \big|E(\omega_{jR}^1) - E(\bar\omega_R^1)\big| \,&\le\, \sup_{0 \le r \le R} 
  \Bigl(|M_j(r) - \overline{M}(r)|\Bigr) \,\int_0^R |\log(r)|\,r \omega_j(r)\dd r \\
  &\quad + \sup_{0 \le r \le R} \Bigl(|\log(r)|\,\overline{M}(r)\Bigr)\,\int_0^R 
  r\bigl|\omega_j(r) - \bar\omega(r)\bigr| \dd r \,\xrightarrow[j \to +\infty]{}\,0\,.
  \label{Econv}
\end{align}
Here we used the convergence of $\omega_j$ to $\bar\omega$ in $L^1(\R^2)$, the a
priori estimates \eqref{omjapriori}, and the fact that $\log(r)\overline{M}(r)$
is bounded as $r \to 0$, as a consequence of \eqref{barombd}. On the other hand,
since the function $-\Phi$ is continuous and bounded from below, and since we
integrate on the bounded domain $\{x \in \R^2\,;\, |x| \le R\}$, we can apply
Fatou's lemma to obtain
\begin{equation}\label{Supconv}
  -S(\bar\omega_R^1) \,=\, \int_{|x|\le R} -\Phi(\bar\omega(x))\dd x \,\le\, 
  \liminf_{j \to +\infty} \int_{|x|\le R} -\Phi(\omega_j(x))\dd x \,=\, 
  -\limsup_{j \to +\infty}S(\omega_{jR}^1)\,.
\end{equation}
Combining \eqref{Econv} and \eqref{Supconv}, we obtain \eqref{propa}. 

We next prove \eqref{propb}. Recalling that $R \ge 1$, we first observe that
\[
  E(\omega_{jR}^2) = -\int_R^\infty M_j(r) \log(r) \,r \omega_j(r)\dd r
  \,\le\, 0\,,
\]
which means that the contribution of $E(\omega_{jR}^2)$ can be disregarded since
we only need an upper bound. The other terms in \eqref{propb} have the following
expressions
\[
  2E(\omega_{jR}^1,\omega_{jR}^2) = -M_j(R)\int_R^\infty \log(r) \,r
  \omega_j(r)\dd r\,, \qquad S(\omega_{jR}^2) \,=\, 2\pi \int_R^\infty
  \Phi(\omega_j(r))\,r \dd r \,.
\]
Since $\omega_j$ is decreasing, we have $\omega_j(r) \le M_j(r)/(\pi r^2) \le M$
for $r \ge R$. So, using Hypotheses~\ref{hypPhi}, we deduce that $\Phi(\omega_j)
\le C_1\omega_j + C_2 \omega_j\log(M/\omega_j)$, where $C_1 \in \R$ and
$C_2 < M/(8\pi)$. It follows that
\begin{equation}\label{IIsplit}
  2E(\omega_{jR}^1,\omega_{jR}^2) + S(\omega_{jR}^2) \,\le\, 2\pi C_1
  \int_R^\infty \omega_j(r) r\dd r + \int_R^\infty \Delta_j(r)\omega_j(r) r\dd r\,,
\end{equation}
where
\[
  \Delta_j(r) \,=\, 2\pi C_2 \log\frac{M}{\omega_j(r)} - M_j(R)\log(r)\,.
\]
In view of \eqref{omjapriori}, the first term in the right-hand side of
\eqref{IIsplit} converges to zero uniformly in $j$ as $R \to +\infty$, and can
therefore be absorbed in the quantity $\delta_1(R)$. To treat the second term,
we fix a positive number $\alpha > 2$ such that $4\pi C_2\alpha \le M$, and we
introduce the mutually disjoints sets
\begin{equation}\label{IIdef}
  I(\alpha,R) \,=\, \bigl\{r \ge R\,;\, \omega_j(r) \,\ge\, M r^{-\alpha}\bigr\}
  \,, \qquad
  I(\alpha,R)^c \,=\, \bigl\{r \ge R\,;\, \omega_j(r) \,< M\, r^{-\alpha}\bigr\}\,.
\end{equation}
As $M_j(R) \ge M/2$, it follows from \eqref{IIdef} that $\Delta_j(r) \le 0$ when
$r \in I(\alpha,R)$, so the last integral in \eqref{IIsplit} can be restricted
to the complement $I(\alpha,R)^c$. But on that set we have the upper bound
$\omega_j(r) < M r^{-\alpha}$, where $\alpha > 2$, and we easily deduce that
$\int_{I(\alpha,R)^c}\Delta_j(r)\omega_j(r) r\dd r$ converges to zero as $R \to
+\infty$, uniformly in $j$. Altogether we obtain \eqref{propb}.

It remains to establish \eqref{propc}, which is an easy task. Indeed 
$\bar\omega$ is a fixed function which satisfies the estimates \eqref{barombd},
so that $2 E(\bar\omega_R^1,\bar\omega_R^2) +\, E(\bar\omega_R^2) \to 0$
as $R \to +\infty$. In addition, we proved in Proposition~\ref{prop:globmax}
that the inegral defining $S(\bar\omega)$ is absolutely convergent, and
this implies that $S(\bar\omega_R^2) \to 0$ as $R \to +\infty$. We thus
obtain \eqref{propc}, and the proof of Theorem~\ref{thm:globmax} is
complete. 
\end{proof}

\begin{ex}\label{ex:algebraic}
We consider the family of algebraic vortices with parameter $\kappa > 1$: 
\[
  \omega(r) \,=\, \frac{1}{(1+r^2)^\kappa}\,, \qquad 
  M \,=\, 2\pi \int_0^\infty r \omega(r)\dd r \,=\, \frac{\pi}{\kappa-1}\,.
\]
The associated stream function $\psi$ satisfies $\psi(r) = \psi(0) + 
\int_0^r \psi'(s)\dd s$ where
\[
  \psi(0) \,=\, \int_0^\infty \log(r)\,\frac{r}{(1+r^2)^\kappa}\dd r\,, 
  \qquad \psi'(r) \,=\, \frac{1}{2(\kappa{-}1)r}\Bigl(1 - \frac{1}{(1+r^2)^{\kappa-1}}
  \Bigr)\,.
\]
We have $\Phi(\omega) = \int_0^\omega \phi(s)\dd s$ where $\phi(\omega(r)) = 
\psi(r)$. Explicitly, for a few values of $\kappa$, we find 
\[
\begin{array}{lll}
  \kappa = \frac32\,:\quad &\DS\psi(r) = \log\bigl(1+\sqrt{1+r^2}\bigr)
  &\DS\phi(\omega) = \log\Bigl(1 + \frac{1}{\omega^{1/3}}\Bigr) \\[3mm] 
  \kappa = 2\,: &\DS\psi(r) = \frac14\log\bigl(1+r^2\bigr) 
  &\DS\phi(\omega) = \frac18 \log\frac{1}{\omega} \\[3mm] 
  \kappa = 3\,: &\DS\psi(r) = \frac18\Bigl(\log\bigl(1+r^2\bigr) - \frac{1}{1+r^2}
  \Bigr)\qquad &\DS\phi(\omega) = \frac{1}{24}\log\frac{1}{\omega} - \frac{
  \omega^{1/3}}{8}
\end{array}
\]
In all cases, we observe that
\[
  \phi(\omega) \,=\, \Phi'(\omega) \,\sim\, \frac{1}{4\kappa(\kappa{-}1)}
  \,\log\frac{1}{\omega} \,=\, \frac{M}{4\pi\kappa} \,\log\frac{1}{\omega}\,,
  \quad \hbox{as}\quad \omega \to 0\,.
\]
It follows that Hypotheses~\ref{hypPhi} are satisfied if and only if $\kappa > 2$. 
\end{ex}

\begin{ex}\label{ex:gaussian}
We next consider the Gaussian vortex $\omega(r) = e^{-r^2/4}$, where $M = 4\pi$. 
In that case we have $\psi(0) = \int_0^{+\infty}\log(r)\,e^{-r^2/4}\dd r = 
2\log(2)-\gamma_E$, so that the stream function satisfies
\[
  \psi(r) \,=\, \psi(0) + \int_0^r \frac{2}{s}\Bigl(1 - e^{-s^2/4}\Bigr)\dd s
  \,=\, 2\log(2)-\gamma_E + \mathrm{E_{in}}(r^2/4)\,,
\]
where
\[
  \mathrm{E_{in}}(z) \,=\, \int_0^z \frac{1-e^{-t}}{t}\dd t \,=\, \sum_{k=1}^\infty
  \frac{(-1)^{k-1}}{k}\,\frac{z^k}{k!}\,, \qquad z \in \C\,.
\]
We conclude that
\[
  \phi(\omega) \,=\, \Phi'(\omega) \,=\,  2\log(2)-\gamma_E + \mathrm{E_{in}}
  \Bigl(\log\frac{1}{\omega}\Bigr)\,.
\]
In particular $\phi(\omega) \sim \log\log\frac{1}{\omega}$ as $\omega \to 0$, 
and Hypotheses~\ref{hypPhi} are satisfied in that case. 
\end{ex}


We do not have much information on the maximizer $\bar\omega$ whose existence
is established in Theorem~\ref{thm:globmax}. We expect that, if $\Phi$ is 
as in Example~\ref{ex:gaussian}, the maximizer is indeed the Gaussian vortex
\eqref{gaussvort}, but except from numerical evidence we have no proof so 
far. Similarly, we believe that the algebraic vortices \eqref{algvort} 
with $\kappa \ge 2$ are global maximizers, but this is known only in the 
particular case $\kappa = 2$, where maximality follows from the logarithmic 
HLS inequality \eqref{logHLS}. 

The examples above also suggest that the decay rate of the maximizer
$\bar\omega(x)$ as $|x| \to \infty$ strongly depends on the behavior of the
function $\Phi(s)$ near $s = 0$. Extending the techniques in the proof of
Theorem~\ref{thm:globmax}, one should be able to prove that, if $\Phi$ is
differentiable to the right at the origin, the corresponding maximizer
$\bar\omega$ is compactly supported. It is also worth mentioning that the
entropy function $\Phi$ associated with any radially symmetric decreasing vortex
$\bar\omega$ through the relation $\bar\psi(x) = \Phi'(\bar\omega(x))$ is
necessarily concave on the range of $\bar\omega$, whereas no concavity 
assumption is included in Hypotheses~\ref{hypPhi}. This suggests that 
the maximizer $\bar\omega$ corresponding to a non-concave function $\Phi$
should be discontinuous, so that its range does not include the intervals 
where $\Phi$ does not coincide with its concave hull. 

\section{Stability of viscous vortices}\label{sec4}

In this final section, we give a new proof of the nonlinear stability of the
Oseen vortices, which are self-similar solutions of the Navier-Stokes in
$\R^2$. Our approach relies on the functional-analytic tools developed in
Section~\ref{sec2}, in connexion with Arnold's variational principle, although
we now consider a dissipative equation for which the Casimir functions
\eqref{I10} are no longer conserved quantities. As in \cite{GW1,GW2}, we
introduce self-similar variables so that Oseen's vortex becomes a stationary
solution, whose stability in thus easier to study. In the new coordinates, 
the viscous vorticity equation takes the form
\begin{equation}\label{SSvort}
  \partial_t \omega + \bigl\{\psi,\omega\bigr\} \,=\, \cL \omega\,, 
  \qquad \Delta \psi \,=\, \omega\,,
\end{equation}
where $\{\psi,\omega\} = \nabla^\perp\psi\cdot \nabla\omega$ and $\cL$ is 
the rescaled diffusion operator
\begin{equation}\label{cLdef}
  \cL \,=\, \Delta_x + \frac{1}{2}\,x\cdot\nabla + 1\,.
\end{equation}
We consider the vortex with Gaussian profile \eqref{gaussvort}, namely
\begin{equation}\label{Gauss2}
  \bar\omega(x) \,=\, \frac{1}{4\pi}\,e^{-|x|^2/4}\,, \qquad 
  \bar u(x) \,=\, \nabla^\perp\bar\psi(x) \,=\, \frac{1}{2\pi}\,
  \frac{x^\perp}{|x|^2}\Bigl(1 - e^{-|x|^2/4}\Bigr)\,.
\end{equation}
It is easy to verify that $\cL\bar\omega = 0$ and $\bigl\{\bar\psi,\bar\omega
\bigr\} = 0$. This implies that $\omega = \alpha \bar\omega$ is a stationary 
solution of \eqref{SSvort} for any $\alpha \in \R$. This family of equilibria
is known to be stable with respect to perturbations in various weighted $L^2$ 
spaces, see \cite{GW2,Ga1}. We present here a new stability proof, which may 
be easier to adapt to more general situations. 

\subsection{Nonlinear stability of Oseen vortices}\label{ssec41}

Given any $\alpha \in \R$, we consider solutions of \eqref{SSvort} of the form 
$\omega = \alpha\bar\omega + \tilde \omega$, $\psi = \alpha\bar\psi + \tilde \psi$. 
The perturbation $\tilde \omega$ satisfies the modified equation
\begin{equation}\label{SSvort2}
  \partial_t \tilde \omega + \alpha \bigl\{\bar\psi,\tilde\omega\bigr\} + 
  \alpha \bigl\{\tilde\psi,\bar\omega\bigr\} + \bigl\{\tilde\psi,\tilde\omega\bigr\} 
  \,=\, \cL \tilde \omega\,, 
\end{equation}
where it is understood that the stream function $\tilde \psi$ is expressed 
in terms of $\tilde\omega$ via the formula \eqref{psidef}, so that 
$\Delta \tilde \psi = \tilde\omega$. We assume henceforth that the perturbation 
$\tilde \omega$ satisfies the moment conditions
\begin{equation}\label{Momcond}
  \int_{\R^2} \tilde \omega \dd x \,=\, 0\,, \qquad \hbox{and}\quad 
  \int_{\R^2} x_j \tilde \omega \dd x \,=\, 0 \quad\hbox{for } j = 1,2\,,
\end{equation}
which are preserved under the evolution defined by \eqref{SSvort2}. As is shown
at the end of Reference~\cite{GW2}, this hypothesis does not restrict the
generality, in the sense that stability with respect to general perturbations
(with no moment conditions) can then deduced by a simple argument. As for the
existence of solutions to \eqref{SSvort2}, we have the following standard
result\:

\begin{lem}\label{Cauchylem}
The Cauchy problem for equation \eqref{SSvort2} is globally well-posed in the 
weighted $L^2$ space $X$ defined by \eqref{Xdef}, where $\cA(x) = 4 |x|^{-2}
(e^{|x|^2/4}-1)$, and the subspace $\cX_1 \subset X$ defined by \eqref{X2def}
is invariant under the evolution. 
\end{lem}

\begin{proof}
It is known that the vorticity equation \eqref{SSvort} or \eqref{SSvort2} 
is globally well-posed in various weighted $L^2$ spaces, see e.g. 
\cite{GW1,Ga1,Ga2}. The ``subgaussian'' weight $\cA$ is not explicitly 
considered in those references, but the arguments therein can be easily 
modified to cover that case too. If $\cA^{1/2}\tilde \omega \in L^2(\R^2)$, then 
all moments of $\tilde \omega$ are well defined, and a direct calculation
shows that the conditions \eqref{Momcond} are preserved under the 
evolution, so that \eqref{SSvort2} is globally well-posed in the 
subspace $\cX_1$. 
\end{proof}

Let $\tilde \omega_0 \in \cX_1$, and let $\tilde \omega \in C^0([0,+\infty),\cX_1)$ 
be the solution of \eqref{SSvort2} with initial data $\tilde \omega_0$. 
By parabolic regularization, we have $\tilde \omega(\cdot,t) \in Z_1 := 
Z \cap \cX_1$ for all $t > 0$, where $Z$ is the weighted Sobolev space
\begin{equation}\label{Zdef}
  Z \,=\, \Bigl\{\omega \in H^1(\R^2)\,;\, \cA^{1/2}\omega \in 
  L^2(\R^2)\,,~\cA^{1/2}\nabla\omega \in L^2(\R^2)\Bigr\}\,.
\end{equation}
For later use, we introduce the following quadratic form on $Z$: 
\begin{equation}\label{Qdef2}
  Q(\omega) \,=\, \int_{\R^2} \Bigl(\cA(x)|\nabla \omega(x)|^2 - 
  \cB(x)\omega(x)^2\Bigr)\dd x\,, \qquad \omega \in Z\,,
\end{equation}
where 
\begin{equation}\label{Bdef}
  \cB \,=\, 1 + \frac12\Bigl(\Delta \cA - \frac{x}{2}\cdot\nabla\cA + 
  \cA\Bigr) \,=\, 1 + \cA - \frac{x\cdot\nabla\cA}{|x|^2}\,.
\end{equation}
We shall verify in Section~\ref{ssecA3} that $\cA/2 \le \cB \le 2\cA$, 
so that the form $Q$ is well defined. 

The following coercivity result plays a crucial role in our argument. 

\begin{thm}\label{Qprop}
The quadratic form $Q$ defined by \eqref{Qdef2} is coercive on the subspace 
$Z_1 = Z \cap \cX_1$\: there exists a constant $\delta > 0$ such that
\begin{equation}\label{Qcoercive}
  Q(\omega) \,\ge\, \delta\int_{\R^2} \cA(x)\omega(x)^2\dd x\,,
  \qquad \hbox{for all }\omega \in Z_1\,.
\end{equation}
\end{thm}

The proof of Theorem~\ref{Qprop} requires a careful analysis, which is postponed
to Section~\ref{ssec42} below. In particular, we shall see that the quadratic
form $Q$ is not not positive on the whole space $Z$, because it takes negative
values on a one-dimensional subspace made of radially symmetric functions. If we
restrict ourselves to functions with zero mean, the form $Q$ is nonnegative but
vanishes on a two-dimensional subspace due to translation invariance. Therefore,
all moment conditions \eqref{Momcond} are necessary to establish the coercivity
of $Q$.

Returning to the solution $\tilde \omega \in C^0([0,+\infty),\cX_1)$ of 
\eqref{SSvort2}, we define for all $t > 0$ the quantities
\begin{equation}\label{JQNdef}
  \begin{split}
  \tilde J(t) \,&=\, \frac{1}{2}\int_{\R^2} \Bigl(\cA(x)\tilde \omega(x,t)^2 + 
  \tilde\psi(x,t)\tilde\omega(x,t)\Bigr)\dd x \,=\, J(\tilde\omega(t))\,, \\
  \tilde Q(t) \,&=\, \int_{\R^2} \Bigl(\cA(x)|\nabla \tilde\omega(x,t)|^2 - 
  \cB(x)\tilde\omega(x,t)^2\Bigr)\dd x \,=\, Q(\tilde\omega(t))\,, \\
  \tilde N(t) \,&=\, \frac{1}{2}\int_{\R^2} \bigl\{\cA(x),\tilde \psi(x,t)\bigr\}
  \tilde \omega(x,t)^2\dd x \,=\, N(\tilde\omega(t))\,.
  \end{split}
\end{equation}
The key observation is: 

\begin{prop}\label{JQprop}
If $\tilde \omega \in C^0([0,+\infty),\cX_1)$ is a solution of \eqref{SSvort2}, the
quantities defined in \eqref{JQNdef} satisfy
\begin{equation}\label{JQNdot}
  \tilde J'(t) \,=\, - \tilde Q(t) - \tilde N(t)\,, \qquad \hbox{for all }
  t > 0\,.
\end{equation}
\end{prop}

\begin{proof}
Using the evolution equation \eqref{SSvort2}, we find
\begin{equation}\label{Jdot}
  \begin{split}
  \tilde J'(t) \,&=\, \int_{\R^2} \Bigl(\cA(x)\tilde \omega(x,t) + \tilde \psi(x,t)
  \Bigr)\partial_t \tilde \omega(x,t)\dd x \\
 \,&=\, \int_{\R^2} \Bigl(\cA\tilde \omega + \tilde \psi\Bigr)\Bigl(
  \cL \tilde \omega - \alpha \bigl\{\bar\psi,\tilde\omega\bigr\} - 
  \alpha \bigl\{\tilde\psi,\bar\omega\bigr\} - \bigl\{\tilde\psi,\tilde\omega\bigr\} 
  \Bigr)(x,t)\dd x\,.
  \end{split}
\end{equation}
We first consider the terms involving the diffusion operator $\cL$ in \eqref{Jdot}. 
We observe that
\begin{equation}\label{Jdot0}
  \int_{\R^2} \tilde \psi(x,t) \,\cL \tilde \omega(x,t)\dd x \,=\, 
  \int_{\R^2} \tilde \omega(x,t)^2 \dd x\,,
\end{equation}
because $\cL \tilde \omega = \Delta\tilde \omega + \frac12 \div(x\tilde \omega)$ 
and
\begin{align*}
  \int_{\R^2} \tilde \psi \,\Delta \tilde \omega \dd x \,&=\, 
  \int_{\R^2} \bigl(\Delta \tilde \psi\bigr) \,\tilde \omega \dd x \,=\, 
  \int_{\R^2} \tilde \omega^2 \dd x\,, \\
  \int_{\R^2} \tilde \psi\,\div(x\tilde \omega) \dd x \,&=\, 
  - \int_{\R^2} \bigl(\Delta\tilde \psi\bigr) \bigl(x \cdot \nabla\tilde\psi
  \bigr)\dd x \,=\, \frac{1}{2}\int_{\R^2} \div\bigl(x |\nabla\tilde\psi|^2\bigr)
  \dd x \,=\, 0\,.
\end{align*}
On the other hand, integrating by parts we obtain by direct calculation
\begin{equation}\label{Jdot1}
  \int_{\R^2} \cA(x)\tilde \omega(x,t)\,\cL \tilde \omega(x,t) \dd x
  \,=\, -Q(\tilde \omega(t)) - \int_{\R^2}\tilde \omega(x,t)^2\dd x\,.
\end{equation}

We next compute the advection terms in \eqref{Jdot}, which are proportional 
to $\alpha$. We claim that
\begin{equation}\label{Jdot2}
  I(\tilde \omega) \,:=\,  \int_{\R^2} \Bigl(\cA\tilde \omega + \tilde \psi\Bigr)\Bigl(
  \bigl\{\bar\psi,\tilde\omega\bigr\} + \bigl\{\tilde\psi,\bar\omega\bigr\}\Bigr)\dd x
  \,=\, 0\,.
\end{equation}
This identity is not surprising, as it means that the quadratic form $J$ is
invariant under the evolution defined by the linearized Euler equation at
$\bar\omega$, see \eqref{I7} for an analogue in the finite-dimensional case.
It can also be verified by direct calculations\: 
\begin{align*}
  &\int_{\R^2} \cA \tilde \omega \bigl\{\bar\psi,\tilde\omega\bigr\} \dd x \,=\, 
  \frac{1}{2}\int_{\R^2} \cA \bigl\{\bar\psi,\tilde\omega^2\bigr\} \dd x 
  \,=\, \frac{1}{2}\int_{\R^2} \bigl\{\cA,\bar\psi\bigr\} \tilde\omega^2\dd x 
  \,=\, 0\,, \\
  &\int_{\R^2} \tilde \psi\bigl\{\tilde\psi,\bar\omega\bigr\}\dd x \,=\,
  \int_{\R^2} \bigl\{\tilde \psi,\tilde\psi\bigr\}\bar\omega\dd x \,=\, 0\,, \\
  &\int_{\R^2} \Bigl(\cA\tilde \omega \bigl\{\tilde\psi,\bar\omega\bigr\}
  + \tilde \psi \bigl\{\bar\psi,\tilde\omega\bigr\}\Bigr)\dd x \,=\, 
  \int_{\R^2} \tilde\omega \Bigl(\cA \bigl\{\tilde\psi,\bar\omega \bigr\}
  + \bigl\{\tilde\psi,\bar\psi \bigr\}\Bigr)\dd x \,=\, 0\,.
 \end{align*}
Here we used the fact that $\bigl\{\cA,\bar\psi\bigr\} = 0$, because $\cA$ and 
$\bar\psi$ are radially symmetric. Moreover,  
\[
  \cA \bigl\{\tilde\psi,\bar\omega \bigr\} + \bigl\{\tilde\psi,\bar\psi \bigr\}
  \,=\, (\nabla \tilde\psi)^\perp \cdot \bigl(\cA \nabla\bar\omega + 
  \nabla\bar\psi\bigr) \,=\, 0\,,
\]
by the very definition of $\cA$. This proves \eqref{Jdot2}. 

Finally, integrating by parts the last term in \eqref{Jdot}, we find
\begin{equation}\label{Jdot3}
  N(\tilde \omega) \,:=\, \int_{\R^2} \Bigl(\cA\tilde \omega + \tilde \psi\Bigr)
  \bigl\{\tilde\psi,\tilde\omega\bigr\}\dd x \,=\, 
  \int_{\R^2} \cA\tilde \omega \bigl\{\tilde\psi,\tilde\omega\bigr\}\dd x \,=\,
  \frac{1}{2}\int_{\R^2} \bigl\{\cA,\tilde\psi\bigr\}\tilde \omega^2 \dd x\,.
\end{equation}
Combining \eqref{Jdot}--\eqref{Jdot3}, we obtain the desired result. 
\end{proof}

To control the nonlinear term $N(\tilde \omega)$, we use the following 
estimate. 

\begin{lem}\label{Nlem}
There exists a constant $C_0 > 0$ such that, for all $\tilde \omega 
\in Z$, the nonlinear term \eqref{Jdot3} satisfies
\begin{equation}\label{Nest}
  |N(\tilde \omega)| \,\le\, C_0\,\|\cA^{1/2}\tilde\omega\|_{L^2}^2 
  \Bigl(\|\cA^{1/2}\tilde\omega\|_{L^2} + \|\cA^{1/2}\nabla\tilde\omega\|_{L^2}
  \Bigr)\,.
\end{equation}
\end{lem}

\begin{proof}
We have $|\{\cA,\tilde\psi\}| \le C |\nabla \cA|\,|\nabla \tilde\psi| \le 
C |x|\cA |\nabla \tilde \psi|$, hence
\[
  |N(\tilde \omega)| \,\le\, C\int_{\R^2} |x|\,|\nabla\tilde\psi| 
  \,\cA\,\tilde \omega^2 \dd x \,\le\, C\,\||x| |\nabla\tilde\psi|\|_{L^\infty}
  \,\|\cA^{1/2}\tilde\omega\|_{L^2}^2\,.
\]
On the other hand, using Proposition~B.1 in \cite{GW1}, H\"older's inequality and 
Sobolev's embedding theorem, we find
\[
  \||x| |\nabla\tilde\psi|\|_{L^\infty} \,\le\, C\bigl(\|\langle x\rangle \tilde 
  \omega\|_{L^{3/2}} + \|\langle x\rangle \tilde \omega\|_{L^3}\bigr) \,\le\, 
  C\bigl(\|\cA^{1/2}\tilde\omega\|_{L^2} + \|\cA^{1/2}\nabla\tilde\omega\|_{L^2}
  \bigr)\,.  
\]
Combining these estimates we arrive at \eqref{Nest}. 
\end{proof}

We are now able to state our final result: 

\begin{thm}\label{Oseenprop}
There exist positive constants $C_1$, $\epsilon_0$, and $\mu$ such that,
for any $\alpha \in \R$ and any $\tilde \omega_0 \in \cX_1$ satisfying 
$\|\tilde\omega_0\|_X \le \epsilon_0$, the solution of \eqref{SSvort2} with 
initial data $\tilde \omega_0$ satisfies
\begin{equation}\label{Oseenest}
  \|\tilde \omega(t)\|_X^2 \,\le\, C_1\,\|\tilde\omega_0\|_X^2
  \,e^{-\mu t}\,, \qquad \hbox{for all }t \ge 0\,. 
\end{equation}
\end{thm}

\begin{proof}
If $\tilde \omega \in C^0([0,+\infty),\cX_1)$ is the solution of \eqref{SSvort2} 
with initial data $\tilde \omega_0$, we define
\[
  m_0(t) \,=\, \|\tilde \omega(t)\|_X^2 \,=\, \|\cA^{1/2}\tilde\omega(t)\|_{L^2}^2
  \quad (t \ge 0)\,, \qquad
  m_1(t) \,=\, \|\cA^{1/2}\nabla\tilde\omega(t)\|_{L^2}^2 \quad (t > 0)\,.
\]
For the Gaussian vortex, we proved in Section~\ref{sec2} that Hardy's inequality 
\eqref{Hardy} holds for some $C_H < 1$. Thus, by Theorems~\ref{Jprop1} and 
\ref{Jprop2}, there exists a constant $\gamma \in (0,1)$ such that
\begin{equation}\label{Jbds}
  \frac{\gamma}{2}\,m_0(t) \,\le\, \tilde J(t) \,\le\, \frac{1}{2}\,m_0(t)\,, 
  \qquad t \ge 0\,. 
\end{equation}
On the other hand, by Theorem~\ref{Qprop}, there exists $\delta > 0$ such that
\begin{equation}\label{Qbds}
  \tilde Q(t) \,\ge\, \delta\,m_0(t)\,, \quad \hbox{and}\quad 
  \tilde Q(t) \,\ge\, m_1(t) - 2 m_0(t)\,, \qquad t > 0\,,  
\end{equation}
where the second inequality follows from the definition \eqref{Qdef2} and the 
inequality $\cB \le 2 \cA$. Taking a convex combination of both estimates
in \eqref{Qbds}, we deduce 
\begin{equation}\label{Qbds2}
  \tilde Q(t) \,\ge\, \mu \bigl(m_0(t) + m_1(t)\bigr)\,, \qquad t > 0\,, 
\end{equation}
where $\mu = \delta/(3+\delta)$. Finally, it follows from Lemma~\ref{Nlem} 
and Young's inequality that
\begin{equation}\label{Nbds}
  |\tilde N(t)| \,\le\, C_0 m_0(t)\bigl(m_0(t)^{1/2} + m_1(t)^{1/2}\bigr) \,\le\,
  \frac{\mu}{4}\bigl(m_0(t) + m_1(t)\bigr) + \frac{2 C_0^2}{\mu}
  \,m_0(t)^2\,.
\end{equation}

Now, as long as $m_0(t) \le \epsilon^2 := \mu^2/(8C_0^2)$, we have by 
\eqref{JQNdot}, \eqref{Jbds}, \eqref{Qbds2}, \eqref{Nbds}
\[
  \tilde J'(t) \,=\, -\tilde Q(t) - \tilde N(t) \,\le\, -\frac{\mu}{2}
  \bigl(m_1(t) + m_0(t)\bigr) \,\le\, -\mu \tilde J(t)\,, 
\]
which implies
\[
  \gamma m_0(t) \,\le\, 2\tilde J(t) \,\le\, 2 \tilde J(0) \,e^{-\mu t} \,\le\, 
  m_0(0)\, \,e^{-\mu t}\,.
\]
As a consequence, if we assume that $\|\tilde\omega_0\|_X^2 = m_0(0) \le \epsilon_0^2 
:= \gamma \epsilon^2$, we have $m_0(t) \le \epsilon^2$ for all $t \ge 0$ and 
estimate \eqref{Oseenest} holds with $C_1 = \gamma^{-1}$. 
\end{proof}

\begin{rem}\label{rem:nonlin}
Except for a slight difference in the definition of the function space 
$X$, Theorem~\ref{Oseenprop} recovers a well known stability result 
for the family of Oseen vortices, see e.g. \cite[Proposition~4.5]{Ga1}. 
The approach originally developed by C.E.~Wayne and the first author
relies on conserved quantities related to symmetries of the problem, 
such as the second order moment $I(\omega)$ in \eqref{MjIdef}. In 
many respects, it is simpler than ours, and it provides an estimate 
of the form \eqref{Oseenest} with explicit constants $C_1$ and $\mu$. 
Note also that, in the limit of large circulation numbers $|\alpha| \to 
\infty$, the enhanced dissipation effect due to fast rotation can be used 
to improve both the decay rate of the perturbations and the size of the basin 
of attraction of the vortex, see \cite{Ga2}. 
\end{rem}

\subsection{Coercivity of the diffusive quadratic form}\label{ssec42}

This section is entirely devoted to the proof of Theorem~\ref{Qprop}, which 
is a key ingredient in Theorem~\ref{Oseenprop}. We first observe that the 
functions $\cA(x),\cB(x)$ in \eqref{Qdef2} are both radially symmetric, 
with radial profiles $A(r), B(r)$ given by the explicit expressions
\begin{equation}\label{ABexp}
  A(r) \,=\, \frac{e^s-1}{s}\,, \qquad 
  B(r) \,=\, \frac{1}{2s^2}\Bigl(e^{s}(1+s)-1-2s\Bigr) + 1\,, \qquad
  s \,=\, \frac{r^2}{4}\,.
\end{equation}
On can also verify that $B/A$ is a decreasing function of $r$ satisfying
$1/2 \le B(r)/A(r) \le 7/4$ for all $r > 0$, see Section~\ref{ssecA3}. 

We next follow a similar approach as in Section~\ref{sec2}. If $\omega \in Z$ is
decomposed in Fourier series like in \eqref{omFourier}, we have
\begin{equation}\label{Qdef3}
  Q(\omega) \,=\, 2\pi \sum_{k \in \Z} \int_0^\infty \biggl\{
  A(r)\Bigl(|\omega_k'(r)|^2 + \frac{k^2}{r^2}\,|\omega_k(r)|^2\Bigr) 
  - B(r)|\omega_k(r)|^2 \biggr\}r\dd r\,,
\end{equation}
and we observe that $\omega \in Z_1$ if and only if
\[
  \int_0^\infty \omega_0(r)\,r\dd r \,=\, 0\,, \qquad 
  \hbox{and}\quad 
  \int_0^\infty \omega_{\pm 1}(r)\,r^2\dd r \,=\, 0\,. 
\]
Introducing the new variables $w_k = A^{1/2}\omega_k \equiv e^\chi \omega_k$, 
where $\chi = \frac12 \log(A)$, we obtain after straightforward calculations
\begin{equation}\label{Qdef4}
  Q(\omega) \,=\, 2\pi \sum_{k \in \Z} \int_0^\infty \Bigl\{
  |w_k'(r)|^2 + \frac{k^2}{r^2}\,|w_k(r)|^2 + W(r)|w_k|^2\Bigr\} r\dd r\,,
\end{equation}
where the potential $W$ is defined by
\begin{equation}\label{Wdef}
  W(r) \,=\, \chi''(r) + \frac{1}{r}\,\chi'(r) + \chi'(r)^2 - 
  \frac{B(r)}{A(r)}  \,=\, \frac{r}{2}\,\chi'(r) - \chi'(r)^2 
  - \frac{1}{2} - e^{-2\chi(r)}\,.
\end{equation}
The coercivity estimate \eqref{Qcoercive} is thus equivalent to the
inequality
\begin{equation}\label{kineq}
  \int_0^\infty \Bigl\{|w_k'(r)|^2 + \frac{k^2}{r^2}\,|w_k(r)|^2 + W(r)|w_k|^2
  \Bigr\} r\dd r \,\ge\, \delta \int_0^\infty |w_k(r)|^2\,r\dd r\,,
\end{equation}
which should hold for all $k \in \Z$ under the conditions
\begin{equation}\label{Z0cond}
  \int_0^\infty w_0(r)\,e^{-\chi(r)}\,r\dd r \,=\, 0\,, \qquad 
  \hbox{and}\quad 
  \int_0^\infty w_{\pm 1}(r)\,e^{-\chi(r)}\,r^2\dd r \,=\, 0\,. 
\end{equation}

For any $k \in \Z$, we denote by $L_k$ the selfadjoint operator in 
$Y = L^2(\R_+,r\dd r)$ defined by
\begin{equation}\label{Lkdef}
  L_k g \,=\, -\frac{1}{r}\partial_r\bigl(r\partial_r g\bigr) + 
  \frac{k^2}{r^2}\,g + W g\,.
\end{equation}
The domain of $L_k$ is exactly the same as for the harmonic oscillator 
in $\R^2$, because the potential $W$ defined by \eqref{Wdef} satisfies
\begin{equation}\label{Wprop}
  W(r) \,>\, \frac{r^2}{16} -\frac{3}{2} \quad \hbox{for all }r > 0\,,
  \quad \hbox{and}\quad
  W(r) \,\sim\, \begin{cases} -3/2 & \hbox{as } r \to 0\,,\\
  r^2/16 & \hbox{as } r \to \infty\,,
\end{cases}
\end{equation}
see Section~\ref{ssecA3}. Our goal is to prove the lower bound $L_k \ge \delta$ 
in the entire space $Y$ when $|k| \ge 2$, and in the subspaces given by 
conditions \eqref{Z0cond} when $k = 0$ or $k = \pm 1$. We consider 
three cases separately. 

\medskip\noindent{\bf Case 1\:} When $|k| \ge 2$, the desired inequality 
is simply obtained by comparing $L_k$ with the usual harmonic operator. 
Indeed, we know from \eqref{Lkdef}, \eqref{Wprop} that 
\begin{equation}\label{Lklow}
  L_k \,>\, -\partial_r^2 -\frac{1}{r}\,\partial_r + \frac{k^2}{r^2}
  +  \frac{r^2}{16} - \frac{3}{2} \,\ge\, \frac{|k|}{2} - 1\,,
\end{equation}
where inequalities are between selfadjoint operators on $Y$. Thus $L_k 
\ge 1/2$ when $|k| \ge 3$, and there exists $\delta > 0$ such that 
$L_k \ge \delta$ when $|k| = 2$, because the inequality in \eqref{Wprop}
is strict. 

\medskip\noindent{\bf Case 2\:} When $|k| = 1$, the lower bound
\eqref{Lklow} is of no use, but it is easy to verify that $L_k \ge 0$
in that case. Indeed, we claim that $L_k g_1 = 0$ where $g_1(r) = e^{\chi(r)}
\,r\,e^{-r^2/4}$. Since $g_1$ is a positive function vanishing at the origin 
and at infinity, this means that $0$ is the lowest eigenvalue of $L_k$ 
in $Y$ when $k = \pm 1$.  To prove the above claim, we first observe that, 
for any (smooth) function $f$ on $\R_+$, we have the identity
\begin{equation}\label{Lkchi}
  \wtL_k f \,:=\, e^\chi L_k (e^\chi f) \,=\, -\frac{1}{r}\partial_r\bigl(r 
  A\partial_r f\bigr) + \frac{k^2}{r^2}\,A f - Bf\,,
\end{equation}
because this is the property we used to go from \eqref{Qdef3} to 
\eqref{Qdef4}. On the other hand, in view of \eqref{psi*def} and
\eqref{Adef}, we have the identity
\begin{equation}\label{omstar}
   -\frac{1}{r}\partial_r\bigl(r A\partial_r \omega_*\bigr) \,=\, 
   \omega_*\,,
\end{equation}
which holds in fact for any vorticity profile $\omega_*$, by definition 
of $A$. In the case of the Lamb-Oseen vortex, if we differentiate 
\eqref{omstar} with respect to $r$, we find that the function 
$f = -2\omega_*' = r\,e^{-r^2/4}$ satisfies the relation
\begin{equation}\label{feq}
   -\frac{1}{r}\partial_r\bigl(r A\partial_r f\bigr) + \frac{1}{r^2}\,A f
   -\Bigl(A'' + \frac{2}{r}\,A' - \frac{r}{2}\,A'\Bigr)f \,=\, f\,.
\end{equation}
But $A'' + 2A'/r -rA'/2 = B -1$ by \eqref{Bdef}, so combining \eqref{Lkchi}
and \eqref{feq} we conclude that $\wtL_k f = 0$ if $|k| = 1$, 
which is the desired result.  

To get coercivity, we now restrict ourselves to the subspace $Y_1 \subset Y$ of
all functions $g$ satisfying $\langle g,h_1 \rangle = 0$ where
$h_1(r) = r\, e^{-\chi(r)}$, see the second relation in \eqref{Z0cond}.  It is
important to observe that $h_1$ is not proportional to $g_1$, so that $Y_1$ is
{\em not} the orthogonal complement in $Y$ of the eigenspace spanned by
$g_1$. However, we have $\langle g_1,h_1 \rangle = 8 \neq 0$, which means that
the closed hyperplane $Y_1$ does not contain the eigenfunction $g_1$.
In view of Remark~\ref{remCoercive} below, we conclude that there exists
some $\delta > 0$ such that $L_k \ge \delta$ on $Y_1$ when $|k| = 1$.

\medskip\noindent{\bf Case 3\:} Finally, we consider the radially
symmetric case where $k = 0$. The difficulty here is that the 
operator $L_0$ is not positive on the entire space $Y$. A numerical
calculation indicates that $L_0$ has one negative eigenvalue $\mu_0 \approx
-0.722$, and that the next eigenvalue $\mu_1 \approx 0.615$ is positive. 
So it is essential to use the first relation in \eqref{Z0cond}, 
and to restrict our analysis to the subspace $Y_0$ of all $g \in Y$ 
such that $\langle g,h_0 \rangle = 0$, where $h_0(r) = e^{-\chi(r)}$. 
Our strategy is to apply Lemma~\ref{LemCoercive} below with $a = 
-\mu_0$, $b = \mu_1$, $\psi = h_0/\|h_0\|$, and $\phi = g_0/\|g_0\|$, 
where $g_0$ denotes a positive function in the kernel of $L_0 - \mu_0$. 
Estimate \eqref{Lcoercive} can be used to prove coercivity of $L_0$ 
on $Y_0$ if we have good lower bounds on the eigenvalues $\mu_0, \mu_1$ 
and on the scalar product $|\langle\phi,\psi\rangle|$, which measures
the angle between the linear spaces spanned by $g_0$ and $h_0$. 

We first estimate the lowest eigenvalue $\mu_0$. We know from the previous 
step that $L_1 g_1 = 0$. Defining $g = c g_1/r = c e^\chi \,e^{-r^2/4}$, 
where $c = (2\log(2))^{-1/2}$ is a normalization constant chosen so that 
$\|g\| = 1$, we deduce that $L_0 g = (2/r)\partial_r g$. This gives the 
relation
\begin{equation}\label{quasimode}
  \Bigl(L_0 + \frac{3}{4}\Bigr)g \,=\, R\,, \qquad \hbox{where}
  \quad R \,=\, \frac{2}{r}\,\partial_r g + \frac{3}{4}\,g \,=\, 
  \Bigl(\frac{3}{4} -\frac{B-1}{A}\Bigr)g\,,
\end{equation}
where we used the identity $(B-1)/A = 1 - A'/(rA) = 1 - 2\chi'/r $, see
\eqref{Bdef}. In Section~\ref{ssecA3} below, we show that $B - 1 < 3A/4$, so
that $R > 0$. This means that the operator $L_0 + \frac{3}{4}$ admits a 
positive supersolution, and using Sturm-Liouville's theory we conclude that
$L_0 + \frac{3}{4} > 0$, so that $\mu_0 > -3/4$. Actually the function $g$ is a
remarkably accurate quasimode, in the sense that the remainder $R$ in
\eqref{quasimode} is small. The norm of $R$ in $Y = L^2(\R_+,r\dd r)$ can be
computed explicitly, see Section~\ref{ssecA4}.  The result is
\begin{equation}\label{normquasi}
  \int_0^\infty R(r)^2 r\dd r \,=\, \frac{1}{16\log(2)}\Bigl(3 - 
  \log(2) - 2\log(\pi)\Bigr)\,,
\end{equation}
so that $\epsilon := \|R\|_Y \approx 0.0396$. Since $L_0$ is 
a selfadjoint operator, we deduce that $L_0$ has an eigenvalue 
in the interval $[-3/4, -3/4+\epsilon]$. Anticipating the fact 
(established below) that $L_0$ has a unique negative eigenvalue, 
we conclude that  $\mu_0 \in [-3/4, -3/4+\epsilon]$.

We next estimate the second eigenvalue $\mu_1$ of $L_0$. It is convenient here
to observe that, if $g = e^{\chi}f$, the relation $L_0 g = \mu g$ is 
equivalent to the generalized eigenvalue problem $\wtL_0 f = \mu A f$, 
where $\wtL_k$ is defined in \eqref{Lkchi}. The second eigenvalue
of that problem is characterized by the inf-sup formula
\begin{equation}\label{infsup}
  \mu_1 \,=\, \inf_{f \in \cF}\,\sup_{r > 0}\,\bigl(\cR[f]\bigr)(r) \,=\, 
  \sup_{f \in \cF}\,\inf_{r > 0}\,\bigl(\cR[f]\bigr)(r)\,,
  \quad \hbox{where} \quad \cR[f] \,=\, \frac{\wtL_0 f}{A f}\,.
\end{equation}
Here $\cF$ denotes the class of all (smooth) functions $f : [0,+\infty) 
\to \R$ such that $f(0) = 1$, $f(r) \to 0$ as $r \to +\infty$, and 
$f$ has exactly one zero in the interval $(0,+\infty)$. Our 
first trial function is $f(r) = e^{-s}(1-\alpha s)$, where $s = r^2/4$ 
and $\alpha = \log(2)^{-1}$. The value of $\alpha$ is chosen so that
the Rayleigh quotient has no singularity\:
\[
  \cR[f] \,=\, \frac{e^{-s}(1+(2{-}\alpha)s+2\alpha s^2) 
  -(1+(1{-}\alpha)s+\alpha s^2)}{2s(1-e^{-s})(1-\alpha s)}\,, 
  \qquad s = \frac{r^2}{4}\,.
\]
It happens that $\cR[f]$ is a decreasing function on $\R_+$, 
with $\cR[f](0) = -3/4+\alpha$ and $\cR[f](+\infty) = 1/2$. 
In view of \eqref{infsup}, this implies that $1/2 < \mu_1 < 
-3/4 + \alpha \approx 0.69$. A better approximation is obtained
using the improved try
\[
  f(r) \,=\, e^{-s}(1-\alpha s)(1+\beta s)\,, \qquad \hbox{where}
  \quad \beta \,=\, \frac{\alpha(1-2\,e^{-1/\alpha})}{2\alpha - 1
  + 2\,e^{-1/\alpha}(1-\alpha)}\,.
\]
If $1/2 < \alpha < \log(2)^{-1}$, then $\beta > 0$ and the Rayleigh quotient has 
no singularity in the interval $(0,+\infty)$. Taking for instance $\alpha = 1.4$ 
gives the excellent lower bound $\mu_1 \ge 0.6$. 

Finally, we use the quasimode $g$ in \eqref{quasimode} and a standard
perturbation argument to estimate the true eigenfunction corresponding to 
the lowest eigenvalue $\mu_0$. We first look for a non-normalized eigenfunction 
of the form $g_0 = g - f$, where $f \perp g_0$. We have
\[
  0 \,=\, (L_0 -\mu_0)g_0 \,=\, (L_0 -\mu_0)g - (L_0 -\mu_0)f \,=\,
  R - \bigl(\mu_0 + {\TS \frac{3}{4}}\bigr)g - (L_0 -\mu_0)f\,,
\]
so that $f = (L_0-\mu_0)^{-1}\bigl(R - (\mu_0 + {\TS \frac{3}{4}})g\bigr)$, 
where $(L_0-\mu_0)^{-1}$ denotes the partial inverse of $L_0 - \mu_0$ on 
its range. The norm of that inverse is bounded by $1/d$, where $d = 
\mu_1 - \mu_0$ is the spectral gap. As $\|R\| = \epsilon$ and $|\mu_0 + 
{\TS \frac{3}{4}}| \le \epsilon$, we conclude that $\|f\| \le 2\epsilon/d$. 
The normalized eigenfunction is 
\[
  \phi \,=\, \frac{g_0}{\|g_0\|} \,=\, \frac{g - f}{\sqrt{1-\|f\|^2}}\,.
\]
Let $\psi = \hat{c} h_0 = \hat{c}\,e^{-\chi}$, where $\hat{c} = \sqrt{3}/\pi$ is a
normalization factor chosen so that $\|\psi\| = 1$. A direct calculation
shows that
\[
  \langle \psi,g\rangle \,=\, c\hat{c} \int_0^\infty e^{-r^2/4}\,r\dd r \,=\, 
  2c\hat{c} \,=\, \frac{1}{\pi}\,\sqrt{\frac{6}{\log(2)}} \,\approx\, 
  0.9365\,,
\]
hence
\begin{equation}\label{phipsi}
  \langle \psi,\phi\rangle \,=\, \frac{\langle \psi,g\rangle - \langle \psi,
  f\rangle}{\sqrt{1-\|f\|^2}} \,\ge\, 2c\hat{c} - \frac{2\epsilon}{d}\,.
\end{equation}
We use Lemma~\ref{LemCoercive} below with $a = -\mu_0 \le 3/4$, $d = 
a+b = \mu_1 - \mu_0 \ge 1.2$, and $\epsilon = \|R\| \le 0.04$. In view 
of \eqref{phipsi}, estimate \eqref{Lcoercive} shows that there exists 
some $\delta > 0$ such that $\langle Lf , f\rangle \ge \delta \|f\|^2$ 
for all $f \in Y_0 = \psi^\perp$. This concludes the proof of 
Theorem~\ref{Oseenprop}. \QED

Finally, we state an elementary lemma that was used twice in the 
above proof. 

\begin{lem}\label{LemCoercive}
Let $X$ be a Hilbert space and $L : D(L) \to X$ be a selfadjoint
operator in $X$. We assume that there exist $\phi \in D(L)$ with
$\|\phi\| = 1$ and $a,b \in \R$ with $a + b\ge 0$ such that\\[1mm]
i) $L\phi = -a \phi$, and \\[1mm]
ii) $\langle Lg , g\rangle \ge b \|g\|^2$ for all $g \in D(L)$ with 
$g \perp \phi$. \\[1mm]
Then, for any $\psi \in X$ with $\|\psi\| = 1$, we have the lower bound
\begin{equation}\label{Lcoercive}
  \langle Lf , f\rangle \ge \Bigl((a+b)|\langle \phi,\psi\rangle|^2 
  - a\Bigr)\|f\|^2\,, \quad \hbox{for all } f \in D(L) \hbox{ with }
  f \perp \psi\,.
\end{equation}
\end{lem}

\begin{proof}
Given $f \in D(L)$, we decompose $f = \langle f,\phi\rangle\phi + g$, 
so that $g \perp \phi$. Since $L\phi = - a \phi$, we find
\[
   \langle Lf , f\rangle \,=\,  \langle Lg , g\rangle - a |\langle f,\phi 
   \rangle|^2 \,\ge\, b\|g\|^2 - a |\langle f,\phi \rangle|^2 \,=\, 
   b\|f\|^2 - (a+b)|\langle f,\phi \rangle|^2\,,
\]
where the inequality follows from ii). We now assume that $f \perp \psi$ 
and decompose $\phi = \langle \phi,\psi\rangle\psi + h$. By Cauchy-Schwarz,
we have
\[
  |\langle f , \phi\rangle|^2 \,=\, |\langle f , h\rangle|^2 \,\le\, 
  \|f\|^2 \,\|h\|^2 \,=\, \|f\|^2 \bigl(1 - |\langle \phi,\psi\rangle|^2
  \bigr)\,,
\]
and combining both inequalities we arrive at \eqref{Lcoercive}. 
\end{proof}

\begin{rem}\label{remCoercive}
In the particular case where $a = 0$ and $b > 0$, the kernel of $L$ is 
one-dimensional, and inequality \eqref{Lcoercive} implies that the 
quadratic form of $L$ is strictly positive on any closed hyperplane 
that does not contain the eigenfunction $\phi$.
\end{rem}

\appendix

\section{Appendix}\label{app}

\subsection{Integral inequalities involving logarithmic weights}
\label{ssecA1}

\begin{proof}[\bf Proof of Proposition \ref{prop:logf}]
Let $B_1 = \{x \in \R^n\,;\, |x| < 1\}$ and $D_M = \bigl\{x \in \R^n\,;\, f(x) 
< M\bigr\}$. To prove \eqref{logf1}, we must verify that
\begin{equation}\label{logf3} 
  \int_{B_1} \Bigl(\log\frac{1}{|x|}\Bigr) \,f(x)\dd x ~\lesssim~ M \,+\, 
  \int_{\R^n \setminus D_M} \Bigl(\log\frac{f(x)}{M}\Bigr)\,f(x)\dd x\,.
\end{equation}
Let $\Omega_1 = \bigl\{x \in B_1\,;\, f(x) \le M|x|^{-n/2}\bigr\}$ and 
$\Omega_2 = \bigl\{x \in B_1\,;\, f(x) > M|x|^{-n/2}\bigr\} \subset 
\R^n\setminus D_M$. We have $B_1 = \Omega_1 \cup \Omega_2$ and
\begin{align*}
  \int_{\Omega_1} \Bigl(\log\frac{1}{|x|}\Bigr) \,f(x)\dd x \,&\le\, M \int_{B_1}
  \frac{1}{|x|^{n/2}}\,\log\frac{1}{|x|}\dd x \,=\, CM\,, \\
  \int_{\Omega_2} \Bigl(\log\frac{1}{|x|}\Bigr) \,f(x)\dd x \,&\le\, \frac{2}{n}
  \int_{\Omega_2}\Bigl(\log\frac{f(x)}{M}\Bigr)\,f(x)\dd x \,\le\, \frac{2}{n}
  \int_{\R^n\setminus D_M} \Bigl(\log\frac{f(x)}{M}\Bigr)\,f(x)\dd x\,,
\end{align*}
hence \eqref{logf3} follows by adding both inequalities. We next consider 
\eqref{logf2}, which reads
\begin{equation}\label{logf4} 
  \int_{D_M} \Bigl(\log\frac{M}{f(x)}\Bigr)\,f(x)\dd x ~\lesssim~ M \,+\, 
  \int_{\R^n\setminus B_1} \bigl(\log|x|\bigr)\,f(x)\dd x\,.
\end{equation}
Let $e = \exp(1)$ and 
\[
  \Omega_3 \,=\, \Bigl\{x \in D_M\,;\, f(x) \le \frac{M}{e(1{+}|x|)^{2n}}\Bigr\}\,,
  \qquad
  \Omega_4 \,=\, \Bigl\{x \in D_M\,;\, f(x) > \frac{M}{e(1{+}|x|)^{2n}}\Bigr\}\,.
\]
Since $t \mapsto t \log(1/t)$ is increasing on $[0,e^{-1}]$ and $s \mapsto 
\log(s)$ is increasing on $[1,+\infty)$, we have
\begin{align*}
  \int_{\Omega_3} \Bigl(\log\frac{M}{f(x)}\Bigr)\,f(x)\dd x \,&\le\, M 
  \int_{\R^n} \frac{1}{e(1{+}|x|)^{2n}} \log\bigl(e(1{+}|x|)^{2n}\bigr)\dd x
  \,=\, C M\,, \\
  \int_{\Omega_4} \Bigl(\log\frac{M}{f(x)}\Bigr)\,f(x)\dd x \,&\le\,  
  \int_{\Omega_4} \log\bigl(e(1{+}|x|)^{2n}\bigr)\,f(x)\dd x \,\le\, CM + 
  2n \int_{\R^n\setminus B_1} \bigl(\log|x|\bigr)\,f(x)\dd x\,,
\end{align*}
and \eqref{logf4} follows in the same way. 

From now on, we assume that $f$ is radially symmetric and nonincreasing in the 
radial direction. In particular, we have for all $x \neq 0$:
\begin{equation}\label{fradsym}
  f(x) \,\le\, \frac{1}{\alpha_n |x|^n}\int_{|y| \le |x|}f(y)\dd y \,\le\, 
  \frac{M}{\alpha_n |x|^n}\,, \qquad \hbox{where}
  \quad \alpha_n \,=\, \frac{\pi^{n/2}}{\Gamma(1 + \frac{n}{2})}\,.
\end{equation}
Since $t \mapsto \log_+(t)$ is increasing, we deduce that
\[
  \int_{\R^n\setminus D_M} \Bigl(\log\frac{f(x)}{M}\Bigr) \,f(x)\dd x \,\le\, 
  \int_{\R^n} \Bigl(\log_+\frac{1}{\alpha_n |x|^n}\Bigr) \,f(x)\dd x \,\le\, 
  C M + n \int_{B_1} \Bigl(\log\frac{1}{|x|}\Bigr) \,f(x)\dd x\,,
\]
which is the converse of \eqref{logf1}. Note that, when $n \le 12$, the 
first term $CM$ in the right-hand side can be dropped, because $\alpha_n
> 1$. In a similar way, we find
\[
  \int_{\R^n\setminus B_1} \bigl(\log|x|\bigr)\,f(x)\dd x \,\le\, 
  \frac{1}{n}\int_{\R^n} \Bigl(\log_+\frac{M}{\alpha_n f(x)}\Bigr)\,f(x)\dd x
  \,\le\, CM + \int_{D_M} \Bigl(\log\frac{M}{f(x)}\Bigr)\,f(x)\dd x\,,
\]
which is the converse of \eqref{logf2}. Again, the first term $CM$ 
in the right-hand side can be dropped when $n \le 12$. 
\end{proof}

\begin{proof}[\bf Proof of Proposition~\ref{prop:Edef}]
Throughout the proof we assume that $M := \|\omega\|_{L^1} > 0$.  We
decompose $E(\omega) = E_1(\omega) + E_2(\omega)$ according to
\[
  E_i(\omega) \,=\, \frac{1}{4\pi}\int_{\Omega_i} \log\frac{1}{|x-y|}
  \,\omega(x)\omega(y)\dd x\dd y\,, \qquad i = 1,2\,,
\]
where $\Omega_1 = \{(x,y) \in \R^2 \times \R^2\,;\,|x-y| < 1\}$ and 
$\Omega_2 = \{(x,y) \in \R^2 \times \R^2\,;\,|x-y| \ge 1\}$. We have to 
verify that the integrals defining the quantities $E_1, E_2$ are convergent 
under assumptions \eqref{omcond}. 

First of all, using inequality \eqref{logf3} above with $n = 2$, we obtain for 
all $x \in \R^2$: 
\[
  \int_{|y-x| < 1} \log\frac{1}{|x-y|}\,|\omega(y)|\dd y \,\le\, 
  C\int_{\R^2} \Bigl(1 + \log_+ \frac{|\omega(y)|}{M}\Bigr)
  |\omega(y)|\dd y\,.
\]
If we multiply both sides by $|\omega(x)|$ and integrate over $x \in \R^2$, we thus 
find
\begin{equation}\label{Ebd1}
 |E_1(\omega)| \,\le\, C M \int_{\R^2} \Bigl(1 + \log_+ \frac{|\omega(y)|}{M}\Bigr)
  |\omega(y)|\dd y\,.  
\end{equation}
On the other hand, we have $\log|x-y| \le \log(|x| + |y|) \le \log(1+|x|)
+\log(1+|y|)$ when $|x-y| \ge 1$. This gives the bound
\begin{equation}\label{Ebd2}
\begin{split}
  |E_2(\omega)| \,&\le\, \frac{1}{4\pi}\int_{\Omega_2} |\omega(x)|\,|\omega(y)| 
  \,\Bigl(\log(1+|x|) + \log(1+|y|)\Bigr)\dd x\dd y \\
  \,&\le\, \frac{M}{2\pi} \int_{\R^2} \,|\omega(y)|\log(1+|y|)\dd y\,.
\end{split}
\end{equation}
Combining \eqref{Ebd1} and \eqref{Ebd2}, we arrive at \eqref{EEbound}.

Finally, we assume that $\omega \in C^1_c(\R^2)$ and $\int_{\R^2} \omega(x)\dd x = 0$. 
The associated stream function $\psi \in C^2(\R)$ defined by \eqref{psidef} satisfies  
$|\psi(x)| = \cO(|x|^{-1})$ and $|u(x)| = |\nabla \psi(x)| = \cO(|x|^{-2})$ as 
$|x| \to \infty$, so that $u \in L^2(\R^2)$. This allows us to integrate by parts
in the first expression \eqref{Edef} of the energy, using the relation $\Delta\psi 
= \omega$, to obtain the elegant formula $E(\omega) = \frac12 \int_{\R^2} |u|^2\dd x$. 
By a density argument, the conclusion remains valid for all integrable vorticities
with zero average satisfying a assumptions \eqref{omcond}.
\end{proof}

\subsection{Positivity of the potential $V$ in some examples}\label{ssecA2}

For the algebraic vortex \eqref{algvort} with $\kappa = 1 + \nu > 1$, 
the potential $V$ defined in \eqref{Vdef} has the following expression
\[
  V(r) \,=\, \frac{1}{r^2(1{+}r^2)^2}\Bigl(3 - 2(\nu{-}1)r^2 + (\nu^2{-}1)r^4
  - 2S - S^2\Bigr)\,, \quad \hbox{where}\quad S \,=\, \frac{\nu r^2}{
  (1+r^2)^\nu - 1}\,.
\]
If $\nu = 1$, then $S = 1$ hence $V \equiv 0$. Otherwise\:

\smallskip\noindent$\bullet$ If $\nu > 1$, we have $(1+r^2)^\nu > 1 + \nu r^2$ 
for all $r > 0$, so that $S < 1$. We deduce 
\begin{equation}\label{Vlow1}
  V(r) \,>\, \frac{\nu-1}{(1{+}r^2)^2}\Bigl(-2 + (\nu{+}1)r^2\Bigr)\,,
\end{equation}
so that $V(r) > 0$ if $r^2 \ge 2/(\nu{+}1)$. In the region where $r^2 
\le 2/(\nu{+}1)$, we use the improved estimate
\begin{equation}\label{Slow}
  S \,=\, \frac{\nu r^2}{(1+r^2)^\nu - 1} \,<\, 1 - \frac{\nu{-}1}{2}\,r^2 
  + \frac{\nu^2{-}1}{12}\,r^4\,, \qquad r > 0\,,
\end{equation}
which can be established by a direct calculation. This gives the lower bound
\begin{equation}\label{Vlow2}
  V(r) \,>\, \frac{(\nu{-}1)r^2}{12(1{+}r^2)^2}\Bigl(5 \nu + 11 + (\nu^2{-}1)r^2
  - \frac{(\nu{-}1)(\nu{+}1)^2}{12}\,r^4\Bigr)\,,
\end{equation}
which implies that $V(r) > 0$ if $(\nu{+}1)r^2 \le 2$. 

\smallskip\noindent$\bullet$ If $0 < \nu < 1$, the calculations are 
entirely similar, except that all inequalities in \eqref{Vlow1}--\eqref{Vlow2}
are reversed. This shows that $V(r) < 0$ in that case. 

\medskip For the Gaussian vortex \eqref{gaussvort}, a direct calculation shows
that
\[
  V(r) \,=\, \frac{3}{4s} - \frac{1}{2} + \frac{s}{4} - \frac{1/2}{e^s - 1} - 
  \frac{s/4}{(e^s -1)^2}\,, \qquad \hbox{where}\quad s \,=\, \frac{r^2}{4}\,. 
\]
Using the lower bound $e^s - 1 \ge s(1+s/2+s^2/6)$, we obtain
\begin{align*}
  V(r) \,&\ge\, \frac{1}{4s}\,\frac{1}{(1{+}s/2{+}s^2/6)^2}\Bigl(
  (3-2s+s^2)(1+s/2+s^2/6)^2 - 2 (1+s/2+s^2/6) - 1\Bigr) \\
  \,&=\, \frac{1}{4}\,\frac{s}{(6{+}3s{+}s^2)^2}
  \Bigl(15 + 12 s + 12 s^2 + 4s^3 + s^4\Bigr) \,>\, 0\,.
\end{align*}

\subsection{Properties of the Gaussian vortex}\label{ssecA3}

Given the expressions of $A,B$ in \eqref{ABexp}, we first verify 
that the ratio $B/A$ is a decreasing function of $r$. We have
\begin{equation}\label{ABaux}
  \frac{B(r)-1}{A(r)} \,=\, \frac12\Bigl(1 + h(r^2/4)\Bigr)\,,
  \qquad \hbox{where}\quad h(s) \,=\, \frac{1}{s} - \frac{1}{e^s - 1}\,.
\end{equation}
Since
\[
  h'(s) \,=\, - \frac{(e^s-1)^2 - s^2e^s}{s^2(e^s-1)^2}
  \,=\, - 4\,e^s\,\frac{\sinh(s/2)^2 - (s/2)^2}{s^2(e^s-1)^2} \,<\, 0\,,
  \qquad s > 0\,,
\]
we see that $h$ is strictly decreasing on $(0,+\infty)$ with 
$h(0) = 1/2$ and $h(+\infty) = 0$. We conclude that $(B-1)/A$, 
hence also $B/A$, is a decreasing function of $r$, and that 
$1/2 \le B/A \le 7/4$. 

We next prove the lower bound \eqref{Wprop} on the potential $W$. 
Since $\chi = \log(A)/2$ with $A$ as in \eqref{gaussvort}, 
a direct calculation shows that the potential $W$ defined by 
\eqref{Wdef} has the following expression
\[
  W(r) \,=\, \frac{s}{4} - \frac{1}{2} - \frac{1}{4s} - \frac{s-1/2}{e^s - 1} - 
  \frac{s/4}{(e^s -1)^2}\,, \qquad \hbox{where}\quad s \,=\, \frac{r^2}{4}\,.
\]
Inequality \eqref{Wprop} is thus equivalent to the positivity of the 
function $G$ defined by
\begin{equation}\label{Gexp}
  G(s) \,=\, 1 - \frac{1}{4s} - \frac{s-1/2}{e^s - 1} - 
  \frac{s/4}{(e^s -1)^2}\,, \qquad s > 0\,.
\end{equation}
If $s \ge 1/2$, we use the lower bound $e^s - 1 \ge s(1+s/2)$ and obtain
\[
  G(s) \,\ge\, \frac{s}{4(2+s)^2}\,\Bigl(7 + 4s\Bigr) \,>\, 0\,.
\]
If $0 < s < 1/2$, the third term in the right-hand side of \eqref{Gexp} 
has the opposite sign. To estimate the denominator, we use the 
upper bound $e^s - 1 \le s(1+s/2)(1+s^2/5)$, which holds for $s \le 1/2$. 
This gives
\[
  G(s) \,\ge\, \frac{s}{4(2+s)^2(5+s^2)}\,\Bigl(27 + 32s + 15s^2 + 4s^3\Bigr) 
  \,>\, 0\,.
\]

\subsection{Computing the norm of the quasimode \eqref{quasimode}}
\label{ssecA4}

In this section we compute the $L^2$ norm of the function $R$ defined 
by \eqref{quasimode}. We recall that $g = c A^{1/2}e^{-r^2/4}$, where 
$c = (2\log(2))^{-1/2}$, and using \eqref{ABaux} we observe that
\[
  \frac{3}{4} -\frac{B-1}{A} \,=\, \frac{1}{4}\,\Bigl(1 - 2h(r^2/4)\Bigr)\,, 
  \qquad \hbox{where} \quad h(s) \,=\, \frac{1}{s} - \frac{1}{e^s -1}\,.  
\]
It follows that
\[
  \|R\|_Y^2 \,=\, \frac{1}{16}\,\int_0^\infty \Bigl(1 - 2h(r^2/4)\Bigr)^2 
  g(r)^2 r\dd r \,=\, \frac{1}{16\log(2)}\,\int_0^\infty \Bigl(1 - 2h(s)\Bigr)^2 
  \,\frac{1}{s}\,\Bigl(e^{-s} - e^{-2s}\Bigr)\dd s\,.
\]
Expanding $(1-2h(s))^2 = 1 - 4h(s) + 4h(s)^2$, we decompose 
\begin{equation}\label{Rexpand}
 \|R\|_Y^2 \,\equiv\, \int_0^\infty R(r)^2 r\dd r \,=\, \frac{I_1 - 4 I_2 + 4 I_3}{
  16\log(2)}\,,
\end{equation}
where the integrals $I_1$, $I_2$, $I_3$ are defined and computed below. 

\smallskip\noindent$\bullet$ Evaluation of $I_1$\:
\[
  I_1 \,=\, \int_0^\infty \frac{1}{s}\,\Bigl(e^{-s} - 
  e^{-2s}\Bigr)\dd s \,=\, \log(2)\,.
\]

\smallskip\noindent$\bullet$ Evaluation of $I_2$\:
\begin{align*}
  I_2 \,&=\, \int_0^\infty \frac{h(s)}{s}\,\Bigl(e^{-s} - e^{-2s}\Bigr)\dd s \\ 
  \,&=\, \int_0^\infty \Bigl(\frac{1}{s} - \frac{1}{e^s -1}\Bigr) 
  \biggl\{\int_0^\infty e^{-st}\dd t\biggr\}\Bigl(e^{-s} - e^{-2s}\Bigr)\dd s \\
  \,&=\, \int_0^\infty \biggl\{\int_0^\infty\Bigl(\frac{1}{s} - \frac{1}{e^s -1}\Bigr)
  \Bigl(e^{-s(1+t)} - e^{-s(2+t)}\Bigr)\dd s\biggr\}\dd t \\
  \,&=\, \int_0^\infty \Bigl(\log\frac{2+t}{1+t} - \frac{1}{2+t}\Bigr)\dd t
  \,=\, (1+t)\log\frac{2+t}{1+t}\,\bigg|_{t=0}^{t=+\infty} \,=\, 1-\log(2)\,.
\end{align*}

\smallskip\noindent$\bullet$ Evaluation of $I_3$\:
\begin{align*}
  I_3 \,&=\, \int_0^\infty \frac{h(s)^2}{s}\,\Bigl(e^{-s} - e^{-2s}\Bigr)\dd s \\
  \,&=\, \int_0^\infty \Bigl(\frac{1}{s} - \frac{1}{e^s -1}\Bigr)^2
  \biggl\{\int_0^\infty ts\,e^{-st}\dd t\biggr\}\Bigl(e^{-s} - e^{-2s}\Bigr) \dd s \\
  \,&=\, \int_0^\infty \biggl\{\int_0^\infty\biggl(\frac{e^{-s(1+t)} - e^{-s(2+t)}}{s}
  - 2\,e^{-s(2+t)} + \frac{s\,e^{-s(2+t)}}{e^s-1}\biggr)\dd s\biggr\}\,t \dd t \\
  \,&=\, \int_0^\infty \Bigl(\log\frac{2+t}{1+t} - \frac{2}{2+t} + \psi^{(1)}(3+t)
  \Bigr)\,t \dd t\,,
\end{align*}
where $\psi^{(1)}$ denotes the trigamma function \cite[Section~6.4]{AS}\:
\[
  \psi^{(1)}(z) \,=\, \int_0^\infty \frac{s\,e^{-sz}}{1 - e^{-s}}\dd s \,=\, 
  \frac{\D^2}{\D z^2}\,\log \Gamma(z)\,, \qquad \Re(z) > 0\,.
\]
It follows that
\begin{align*}
  I_3 \,&=\, \frac{t^2+4}{2}\,\log(2+t) - \frac{t^2-1}{2}\,\log(1+t) - 
  \frac{3t}{2} + t\bigl(\log \Gamma\bigr)'(3+t) - \bigl(\log \Gamma\bigr)(3+t)  
  \,\bigg|_{t=0}^{t=+\infty} \\
  \,&=\, \frac{1}{4}\Bigl(7 - 6 \log(2) - 2\log(\pi)\Bigr)\,.
\end{align*}
Here we used Stirling's formula to compute an asymptotic expansion of 
$\bigl(\log \Gamma\bigr)(3+t)$ and its derivative as $t \to +\infty$. 
Inserting the values of $I_1$, $I_2$, $I_3$ into \eqref{Rexpand}, we 
arrive at \eqref{normquasi}. 

\subsection{The Poisson structure on $\cP$}\label{ssecA5}

For two functions $\phi,\psi$ on $\R^2$ we use the familiar notation
$\{\phi,\psi\} = \partial_1\phi\partial_2\psi-\partial_2\psi\partial_2\phi$.
Now, if $\cF$ and $\cG$ are two functionals of $\cP$, the standard way of
defining their Poisson bracket is
\begin{equation}\label{PB}
  \bm{\{}\cF,\cG\bm{\}}(\om) = -\irt \om \left\{\frac{\delta\cF}{\delta\om},
  \frac{\delta \cG}{\delta\om}\right\}\dd x\,,
\end{equation}
where $\frac{\delta\cF}{\delta\om}$ is the usual ``variational derivative'' of
$\cF$, namely, the function on $\R^2$ defined by the relation
\[
  \Bigl(\frac{\dd}{\dd\ve}\cF(\om+\ve\eta)\Bigr)\Big|_{\ve=0} \,=\, \irt 
  \frac{\delta\cF}{\delta\om}(x) \eta(x)\dd x\,,
\]
for all (smooth and compactly supported) increments $\eta$. In particular, 
the variational derivative of the energy functional \eqref{Edef} is 
$\frac{\delta E}{\delta\om} = -\psi$, where $\psi$ is the stream function
\eqref{psidef}, and Euler's equation $\partial_t\omega + \{\psi,\omega\} = 0$ 
can therefore be written in the ``canonical'' form $\partial_t\omega = \bm{\{}
E,\omega\bm{\}}$.

\bigskip\noindent
{\bf Thierry Gallay}\\ 
Institut Fourier, Universit\'e Grenoble Alpes, 100 rue des Maths, 38610 Gi\`eres, 
France\\
Email\: {\tt Thierry.Gallay@univ-grenoble-alpes.fr}

\bigskip\noindent
{\bf Vladim\'ir \v{S}ver\'ak}\\
School of Mathematics, University of Minnesota\\
127 Vincent Hall, 206 Church St.\thinspace SE, Minneapolis, MN 55455, USA\\
Email\: {\tt sverak@math.umn.edu}


\begin{thebibliography}{99}
\setlength{\itemsep}{-0.4mm}

\vspace{-0.2cm}

\bibitem{AS} M. Abramowitz and I. Stegun,  
{\em Handbook of Mathematical Functions with Formulas, Graphs, and 
Mathematical Tables}, Dover, 1964. 

\bibitem{Arn1} V. I. Arnold, Conditions for nonlinear stability of 
stationary plane curvilinear flows of an ideal fluid, 
Dokl. Acad. Nauk SSSR {\bf 162} (1965), 975--978. 

\bibitem{Arn2} V. I. Arnold, Sur un principe variationnel pour les \'ecoulements 
stationnaires des liquides parfaits et ses applications aux probl\`emes de 
stabilit\'e non lin\'eaire, J. de M\'ecanique {\bf 5} (1966), 29--43.

\bibitem{Arn3} V. I. Arnold, Sur la g\'eom\'etrie diff\'erentielle des groupes 
de Lie de dimension infinie et ses applications \`a l'hydrodynamique des fluides
parfaits, Ann. Inst. Fourier {\bf 16} (1966), 319--361.

\bibitem{Arnold-Khesin} V.~I.~Arnold,  B.~A.~Khesin; {\em Topological methods 
in hydrodynamics}, Second edition. Applied Mathematical Sciences, 125, Springer

\bibitem{Burton} 
G.~R.~Burton, Global nonlinear stability for steady ideal fluid flow in
bounded planar domains. Arch. Ration. Mech. Anal. 176 (2005), no. 2, 149--163.

\bibitem{Carlen-Loss}
E.~Carlen, M.~Loss, Competing symmetries, the logarithmic HLS inequality and 
Onofri's inequality on $S^n$. Geom. Funct. Anal. 2 (1992), no. 1, 90--104.

\bibitem{CL} E. A. Coddington and N. Levinson,
{\em Theory of Ordinary Differential Equations}, McGraw-Hill, 1955. 

\bibitem{Dei} K. Deimling, {\em Nonlinear Functional Analysis}, Dover, 1985. 

\bibitem{Ga1} Th. Gallay,
Stability and interaction of vortices in two-dimensional viscous flows", 
Discr. Cont. Dyn. Systems Ser. S {\bf 5} (2012), 1091--1131. 

\bibitem{Ga2} Th. Gallay, 
Enhanced dissipation and axisymmetrization of two-dimensional viscous vortices, 
Arch. Rational Mech. Anal. {\bf 230} (2018), 939--975. 

\bibitem{GW1} Th. Gallay and C. E. Wayne,
Invariant manifolds and the long-time asymptotics of the
Navier-Stokes and vorticity equations on {$\R^2$},
Arch. Ration. Mech. Anal. {\bf 163} (2002), 209--258. 

\bibitem{GW2}
Th. Gallay and C.E. Wayne,
Global stability of vortex solutions of the two-dimensional 
Navier-Stokes equation, Commun. Math. Phys. {\bf 255} 
(2005), 97--129.

\bibitem{Ha} P. Hartman,
{\em Ordinary Differential Equations}, John Wiley \& Sons, New York, 1964. 

\bibitem{Ka} T. Kato,
{\em Perturbation Theory for Linear Operators}, Grundlehren der mathematischen
Wissenschaften {\bf 132}, Springer, 1966. 

\bibitem{LL} E. Lieb and M. Loss,
{\em Analysis}, Graduate Studies in Mathematics {\bf 14},
AMS, Providence, 1997. 

\bibitem{Lions} P.-L. Lions, 
The concentration-compactness principle in the calculus of variations. The 
locally compact case, part 1, Annales de l’IHP section C {\bf 1} (1984), 109--145. 

\bibitem{MP} C. Marchioro and M. Pulvirenti, 
{\em Mathematical Theory of Incompressible Nonviscous Fluids}, 
Applied Mathematical Sciences {\bf 96}, Springer, 1994.

\bibitem{Maz} V. Maz'ya, {\em Sobolev spaces}, 
Grundlehren der mathematischen Wissenschaften {\bf 342},
Springer, 2011. 

\bibitem{Onofri}
E.~Onofri, On the positivity of the effective action in a theory of random surfaces. 
Comm. Math. Phys. 86 (1982), no. 3, 321--326.

\bibitem{Onsager}
L.~Onsager, Statistical hydrodynamics, Nuovo Cimento (9) 6 (1949), Supplemento, 
no. 2 (Convegno Internazionale di Meccanica Statistica), 279--287. 

\bibitem{RS4} M. Reed and B. Simon, {\em Methods of Modern Mathematical 
Physics IV: Analysis of Operators}, Academic Press, 1978. 

\end{thebibliography}
\end{document}